\documentclass[12pt]{amsart}

\usepackage{amssymb,amsthm,amsmath}

\RequirePackage[dvipsnames,usenames]{color}
\input{kmacros3.sty}
\usepackage{mabliautoref}

\usepackage{tikz}
\usepackage{tikz-cd}
\usetikzlibrary{cd}
\usepackage{graphicx}
\usepackage[all,cmtip]{xy}

\numberwithin{equation}{theorem}

\usepackage{bm}
\usepackage{pifont}
\usepackage{upgreek}

\usepackage{eucal}
\usepackage{ulem}
\usepackage{stmaryrd}
\usepackage{xspace}

\usepackage{fullpage}

\usepackage{enumerate}
\usepackage{calc}

\usepackage{verbatim}
\usepackage{alltt}
\usepackage{scalerel}

\def\todo#1{\textcolor{red}%
{\footnotesize\newline{\color{red}\fbox{\parbox{\textwidth-15pt}{\textbf{todo: } #1}}}\newline}}

\def\commentbox#1{\textcolor{red}%
{\footnotesize\newline{\color{red}\fbox{\parbox{\textwidth-15pt}{\textbf{comment: } #1}}}\newline}}

\newcommand{\BCM}{\textnormal{BCM}\xspace}
\newcommand{\BCMTight}{{\textnormal{BCM*}}\xspace}
\newcommand{\perf}{\textnormal{perf}}

\renewcommand{\m}{\mathfrak{m}}

\begin{document}

\title{BCM-thresholds of non-principal ideals}
\author{Sandra Rodr\'iguez-Villalobos}
\address{Department of Mathematics, University of Utah, Salt Lake City, UT 84112, USA}
\email{rodriguez@math.utah.edu}
\author{Karl Schwede}
\address{Department of Mathematics, University of Utah, Salt Lake City, UT 84112, USA}
\email{schwede@math.utah.edu}
\maketitle
\begin{abstract}
    Generalizing previous work of the first author, we introduce and study a characteristic free analog of the $F$-threshold for non-principal ideals, BCM-thresholds.  We show that this coincides with the classical $F$-threshold for weakly $F$-regular rings and that the set of BCM-thresholds coincides with the set of BCM-jumping numbers in a complete local regular ring.  We obtain results on $F$-thresholds of parameter ideals analogous to results of Huneke-\mustata-Takagi-Watanabe as well as a mixed characteristic version of one of their results on multiplicity.
    Instead of taking ordinary powers of an ideal, our definition uses fractional integral closure in an absolute integral closure of our ambient ring.
\end{abstract}

\section{Introduction}

Suppose $R$ is a ring of characteristic $p > 0$ and $\fra, J$ are proper nonzero ideals with $\fra \subseteq \sqrt{J}$.  Then the $F$-threshold of $\fra$ with respect to $J$ is a measure of how powers of $\fra$ compare with Frobenius powers of $J$.  That is, it is the limit 
\[
    c^J(\fra) := \lim_{e \to \infty} { \max\{n \;|\; \fra^n \nsubseteq J^{[p^e]} \}\over p^e}
\]
It is not obvious that this limit exists, but it does, a fact proved in full generality in  \cite{DeStefaniNunezBetancourtPerezFThresholdsAndRelated}.  
The $F$-threshold was introduced in \cite{MustataTakagiWatanabeFThresholdsAndBernsteinSato} as a generalization of the $F$-pure threshold ($\fpt(\fra)$) with which it agrees when $R$ is regular local and $J$ is maximal (the $F$-pure threshold is an analog of the log canonical threshold \cite{TakagiWatanabeFPureThresh}).  In the case of a regular local $R$, as one varies $J$, one obtains that the set of different $F$-thresholds is exactly the set of jumping numbers of $\tau(R, \fra^t)$.  $F$-thresholds also have intriguing interpretations for parameter ideals and provide insight into the multiplicity of various quotient rings \cite{HunekeMustataTakagiWatanabeFThresholdsTightClosureIntClosureMultBounds}. Other recent work on F-thresholds includes \cite{Trivdei.FThresholdsForProjectiveCurves,BadillaCespedesNunezBetantcourtRodriguezVillalobos.FVolumes,GonzalezJaramillo-VelezNunezBetancourt.FThresholdsAndTestIdealsOfThomSebastiani,JeffriesNunezBetancourtQuinlanGallego.BSTheoryForSingularPositive,DeStefaniNunEzBetancourtSmirnov.DefectOfFPT}. 

In \cite{RodriguezBCMThresholdsHypersurfaces}, the first author generalized some of the results of \cite{MustataTakagiWatanabeFThresholdsAndBernsteinSato} to arbitrary complete local domains in  the case when $\fra = (f)$ was principal.  For a fixed balanced big Cohen-Macaulay (BCM) $R^+$-algebra $B$, one should look at 
\[
    c^J_B(f) = \sup \{ t \in \bQ \;|\; f^t \notin JB \}.
\]
This agrees with $c^J(f)$ in characteristic $p > 0$ when $R$ is strongly $F$-regular \cite[Proposition 2.0.4, Definition 3.0.1]{RodriguezBCMThresholdsHypersurfaces} (and coincides with some related invariants in general), and assuming $R$ is regular, the set of such numbers agree with the jumping numbers of the BCM-test ideal $\tau_B(R, f^t)$ in general where $\tau_B$ is as defined in \cite{MaSchwedeSingularitiesMixedCharBCM}.

In this article, we explore characteristic free BCM-Thresholds when $\fra$ is not principal.  The first question then becomes how to replace $f^t$.  We consider $(\fra R^+)_{>t}$, the $(>t)$th fractional integral closure of $\fra$ in $R^+$, an absolute integral closure of $R$.  In other words, 
\[
    (\fra R^+)_{>t} = \{x \in R^+ \;|\; v(x) > t v(y) \text{ for each valuation of $K(R^+)$ over $R^+$ and $y \in \fra R^+$}\}
\]
Such fractional integral closures have appeared in a number of contexts although primarily for Noetherian rings inasmuch as we are aware, see for instance \cite[Section 10.5]{HunekeSwansonIntegralClosure}.  Because of this, we develop some of the basic theory in the non-Noetherian setting, see \autoref{subsec.FractionalIntegrallyClosedPowers} and \autoref{subsec.RelationWithIntegralClosureAndFrationalPowers}.

We now define
\[
    c^J_B(\fra) = \sup \{t \in \bQ \;|\; (\fra R^+)_{>t} \nsubseteq JB\}
\]
the \emph{\BCM-threshold of $\fra$ with respect to $J$ along $B$}, see \autoref{def.BCMThresholdGeneral}.  Thanks to \autoref{prop.FinitenessAssumingBS}, if $\fra \subseteq \sqrt{J}$, then $c^J_B(\fra)$ is finite as long as $B$ satisfies some technical conditions (which the usual big Cohen-Macaulay algebras always do, see \autoref{def.BCMThresholdGeneral}).

We are able to show that this definition satisfies some of the basic properties alluded to above.  First, we show it typically coincides with the classical $F$-threshold.
\begin{theoremA*}[{\autoref{thm.FThresholdTheoremViaR+Theorem}, equation \autoref{eq.ValuativeDescriptionOfcStarJ}, and \autoref{cor.FinalDescriptionOfcstar}}]
    Suppose $R$ is a complete local Noetherian $F$-finite domain of characteristic $p > 0$, $\fra, J \subseteq R$ are proper nonzero ideals and $\fra \subseteq \sqrt{J}$.  
    If $R$ is weakly $F$-regular, then $c^J_B(\fra)$ coincides with the classical $F$-threshold $c^J(\fra)$ as long as $B$ is large enough to capture tight closure (\autoref{def.BCMTight}).
\end{theoremA*}
Even without the $F$-regularity hypothesis, $c^J_B(\fra)$ still coincides with a natural $F$-threshold-like-invariant, the $F$-threshold up to tight closure in the sense of \cite[Definition 8.5]{MengMukhopadhyay.hFunctionHKDensityFrobeniusPoincare}, and which we denote by $c_*^J(\fra)$ (see also \cite[Theorem 8.10]{MengMukhopadhyay.hFunctionHKDensityFrobeniusPoincare} for other cases when $c^J(\fra) = c^J_*(\fra)$).

Analogous to characteristic $p > 0$, the \BCM-thresholds coincide with a variant of \BCM-jumping numbers.  To explain this, we first define a variant of the \BCM-test ideal of a pair $(R, \fra^t)$ with respect to a BCM $R^+$-algebra $B$.  We denote this by $\tau_{B,\elt}(R, (\fra R^+)_{>t})$.  In positive characteristic, this coincides with $\tau(R, \fra^t)$ and in mixed characteristic it coincides with previous BCM test ideals $\tau_B(R, \fra^t)$ as long as $B$ is perfectoid and sufficiently large, see \autoref{thm.BCMTestIdealViaValuative}.

\begin{theoremB*}[{\autoref{Main}}]
    Suppose $R$ is a complete local regular Noetherian domain and $B$ is a $\fram$-adically complete big Cohen-Macaulay $R^+$-algebra, then the set of BCM-thresholds $c^J_B(\fra)$ coincides with the set of jumping numbers of $\tau_{B, \elt}(R, (\fra R^+)_{>t})$:  those positive $t$ such that $\tau_{B,\elt}(R, (\fra R^+)_{>t}) \neq \tau_{B,\elt}(R, (\fra R^+)_{>t-\epsilon})$ for $1 \gg \epsilon > 0$.  
    
    In particular, in positive characteristic, or in mixed characteristic if $B$ is perfectoid and sufficiently large, this coincides with previous notions of jumping numbers of BCM test ideals.
\end{theoremB*}

In the case that $J$ is a parameter ideal, we obtain the following analog of \cite[Theorem 3.3]{HunekeMustataTakagiWatanabeFThresholdsTightClosureIntClosureMultBounds}.

\begin{theoremC*}[\autoref{thm.BCMVersionOfHMTWForParameter}]
    Let $R$ be a complete Noetherian local domain and $B$ BCM $R^+$-algebra
    satisfying the Brian{\c{c}}on-Skoda property.  Fix $x_1,\ldots, x_n$ a
    partial system of parameters of $R$ and set $J=(x_1,\ldots, x_n)$.
    Given an ideal $I\supseteq J$, $c^J_B(I)=n$ if and only if $\overline{I} = \overline{J}.$  
  \end{theoremC*}

  The Brian{\c{c}}on-Skoda property mentioned above is held by all big Cohen-Macaulay algebras satisfying a sufficiently good weak functoriality condition, see \cite{RodriguezSchwedeBrianconSkodaViaWeakFunctoriality}.

  Finally, we recover a mixed characteristic version of a result of Huneke-\mustata-Takagi-Watanabe.  We note that a substantial  part of this proof was provided to us by Linquan Ma, we thank him for letting us include it.

  \begin{theoremD*}[{\autoref{thm.MultiplicityStatement} \cf \cite[Theorem 5.6]{HunekeMustataTakagiWatanabeFThresholdsTightClosureIntClosureMultBounds}}]    
Suppose $(R, \fram = (x_1, \dots, x_d))$ is a  $d$-dimensional regular local ring  which is essentially finite type algebra over a DVR of mixed characteristic $(0, p> 0)$  and $J = (x_1^{a_1}, \dots, x_d^{a_d}) \subseteq R$ for some $a_i > 0$.  If $\fra$ is $\fram$-primary, then 
    \[
        e(\fra) \geq \Big( {d \over c^{J\widehat{R}}_B(\fra \widehat{R})} \Big)^d e(J)
    \]
    for $B$ a sufficiently large perfectoid \BCM $\widehat{R}^+$-algebra.  
\end{theoremD*}

\subsection*{Acknowledgements}
Sandra Rodr\'iguez-Villalobos was supported by  NSF Grant DMS-2101800.  Karl Schwede was supported by NSF Grants \#2501903, \#2101800 and NSF FRG Grant DMS-1952522 as well as Simons Travel Support for Mathematicians SFI-MPSTSM-00013051.  This material is based upon work supported by the National Science Foundation under Grant No. DMS-1928930, while the authors were in residence at the Simons Laufer Mathematical Sciences Institute (formerly MSRI) in Berkeley, California, during the Spring 2024 semester.  The authors thank Rankeya Datta, Alessandro De Stefani, Mark Johnson, Linquan Ma, Jonathan Monta\~no, Ilya Smirnov, and Kevin Tucker for valuable conversations.  We thank Neil Epstein, Linquan Ma, and Ilya Smirnov for valuable comments on a previous draft of the paper.  We particularly thank the referee for numerous corrections, requests for clarification, and useful suggestions.  Finally we thank Linquan Ma for help with Theorem D.

\section{Background}

We begin by recording the notion of the Frobenius ($F$-)thresholds and two variants.  For even further generalizations, see for instance \cite{DeStefaniNunezBetancourtPerezFThresholdsAndRelated}.  In what follows, for an ideal $J \subseteq R$ in a Noetherian ring of characteristic $p > 0$, we use $J^{F}$ to denote Frobenius closure and $J^*$ to denote tight closure, see \cite{HochsterHunekeTC1}.

\begin{definition}[$F$-thresholds and variants]
    Suppose $R$ is a Noetherian domain of characteristic $p > 0$ and $J, \fra \subseteq R$ are proper nonzero ideals with $\fra \subseteq \sqrt{J}$.  
    
    \begin{description}
    \item[Classical $F$-thresholds \cite{MustataTakagiWatanabeFThresholdsAndBernsteinSato}]    We define $\nu_\fra^J(p^e) := \max\{ n > 0 \;|\; \fra^n \nsubseteq J^{[p^e]} \}$.  With that fixed, we define the \emph{$F$-threshold of $\fra$ with respect to $J$} to be: 
    \[
        c^J(\fra) = \lim_{e \to \infty} {\nu_\fra^J(p^e)\over p^e }.
    \]

    \item[$F$-thresholds up to Frobenius closure \cite{HunekeMustataTakagiWatanabeFThresholdsTightClosureIntClosureMultBounds}]
    We define $\tld \nu_\fra^J(p^e) := \max\{ n > 0 \;|\; \fra^n \nsubseteq (J^{[p^e]})^F \}$ 
    and define the \emph{perfect $F$-threshold of $\fra$ with respect to $J$} to be:
    \[
        \tld c^J(\fra) = \lim_{e \to \infty} {\tld \nu_\fra^J(p^e)\over p^e }
    \]
    
    \item[$F$-thresholds up to tight closure ]
    Set $\nu_{\fra}^{*J}(p^e) := \max\{ n > 0 \;|\; \fra^n \nsubseteq (J^{[p^e]})^* \}$ and define the \emph{tight closure $F$-threshold of $\fra$ with respect to $J$} to be:
    \[
        c_*^{J}(\fra) = \lim_{e \to \infty} {\nu_{\fra}^{*J}(p^e)\over p^e }.
    \]
    \end{description}
\end{definition}

If $R$ is $F$-pure, the first and second limits agree.  If $R$ is weakly $F$-regular, all three agree.

It is not obvious that these limits exist.
The $F$-threshold limit was shown to exist in full generality in  \cite{DeStefaniNunezBetancourtPerezFThresholdsAndRelated}.  The perfect $F$-threshold limit was shown to exist earlier in \cite[Page 6]{HunekeMustataTakagiWatanabeFThresholdsTightClosureIntClosureMultBounds}.  
The $F$-threshold up to tight closure was defined in \cite[Definition]{MengMukhopadhyay.hFunctionHKDensityFrobeniusPoincare} where it was denoted by $r_{R,J,\fra}$.  The limit was shown to exist in \cite[Lemma 8.6]{MengMukhopadhyay.hFunctionHKDensityFrobeniusPoincare}.  Their proof, though stated in the local case, extends to the non-local case.  

It is clear that all three limits coincide when $R$ is weakly $F$-regular.  For other cases when they coincide see \cite[Theorem 8.10]{MengMukhopadhyay.hFunctionHKDensityFrobeniusPoincare}.

\subsection{Fractional integrally closed powers}
\label{subsec.FractionalIntegrallyClosedPowers}

    We recall the notion of fractional integrally closed powers of ideals.  One reference is \cite[Section 10.5]{HunekeSwansonIntegralClosure} also see \cite{Knutson.BalancedNormalConesAndFultonMacPhersonIntersectionTheory,Lejeune-JalabertTeissier.IntegralClosureAndEquisingularity,UlrichHong.SpecializationAndIntegralClosure,BisuiDasHaMontano.RationalPowersInvariantIdeals,GoelMaddoxTaylor.NewVersionsOfFrobeniusAndIntegralClosureOfIdeals} for related discussion.  We will eventually need to generalize this notion outside of the Noetherian setting in a way that we believe experts already knew.
    In what follows, if $R$ is an integral domain, we say a valuation $v$ of $K(R)$ is \emph{over $R$} if $v$ is nonnegative on $R$.  We notice that if $J = (f_1, \dots, f_n)$ is an ideal and $v$ is any valuation over $R$ (possibly nondiscrete), then 
    \begin{equation}
        \label{eq.ValuationOnFinitelyGeneratedIdealIsMin}
        v(J) :=\inf \{ v(x) \;|\; x \in J \}= \min\{ v(f_1), \dots, v(f_n) \} = \min \{ v(x) \;|\; x \in J \} .
    \end{equation}
    As any $x \in J$ can be written as a $R$-linear combination of the $f_i$, \cf \cite[discussion after Definition 6.8.9]{HunekeSwansonIntegralClosure}.

    We recall the following from essentially \cite[Section 10.5, Proposition 10.5.2(6)]{HunekeSwansonIntegralClosure}.  

    \begin{definition}[{\cite[Proposition 10.5.2]{HunekeSwansonIntegralClosure}}]
        \label{def.FractionalPowerInNoetherianRing}
        Suppose $R$ is a Noetherian domain and $I \subseteq R$ is an ideal.  For any {real} number $t \geq 0$ we define the ideal  
        \[
            I_t := \{ x \in R \;|\; v(x) \geq t v(I) ;\; \text{ where $v$ runs over discrete rank-1 valuations of $K(R)$ over $R$}\}
        \]        
        Likewise, we define 
        \[
            I_{>t} := \{ x \in R \;|\; v(x) > t v(I) ;\; \text{ where $v$ runs over discrete rank-1 valuations of $K(R)$ over $R$}\}.
        \]        
    \end{definition}

        If $t = a/b$ with $a, b \in \bZ_{\geq 0}$, $b \neq 0$, then 
    \begin{equation}
        \label{eq.RationalPowerIdealPowers}
        I_t = \big\{x \in R \;|\; x^b \in \overline{I^a} \big\}.
    \end{equation}
    Indeed, \eqref{eq.RationalPowerIdealPowers} is the definition provided in \cite{HunekeSwansonIntegralClosure}, but we will typically use the valuative description.

    We briefly describe some alternate characterizations of $I_t$.

    \begin{lemma}
        \label{lem.FractionalPowersNoetherianAlternateDescriptions}
        With notation as in \autoref{def.FractionalPowerInNoetherianRing}, then the following are equivalent for any $x \in R$.
        \begin{enumerate}
            \item $x \in I_t$. \label{lem.FractionalPowersNoetherianAlternateDescriptions.a}
            \item $x^b \in \overline{I^a}$ for some rational number $a/b\geq t$ \label{lem.FractionalPowersNoetherianAlternateDescriptions.b}
            \item $v(x) \geq t v(I)$ for all Rees valuations $v$ of $I$. \label{lem.FractionalPowersNoetherianAlternateDescriptions.c}
            \item {If $t$ is rational}, then for each valuation (possibly nondiscrete) $v$ of $K(R)$ over $R$, there exists $y \in I$ such that $v(x) \geq t v(y)$. \label{lem.FractionalPowersNoetherianAlternateDescriptions.d}
            \item If $R$ is excellent {and $t$ is rational}, then $v(x) \geq t v(I)$ for all divisorial valuations of $K(R)$ over $R$. \label{lem.FractionalPowersNoetherianAlternateDescriptions.e}
        \end{enumerate}
        {
        Similarly, the following are equivalent for any $x \in R$.
        \begin{enumerate}
            \setcounter{enumi}{5}
            \item $x \in I_{>t}$. \label{lem.FractionalPowersNoetherianAlternateDescriptions.f}
            \item $x^b \in \overline{I^a}$ for some rational number $a/b>t$. \label{lem.FractionalPowersNoetherianAlternateDescriptions.g}
            \item $v(x) > t v(I)$ for all Rees valuations $v$ of $I$. \label{lem.FractionalPowersNoetherianAlternateDescriptions.h}
        \end{enumerate} }
    \end{lemma}
    \begin{proof}
        {If $x^b \in \overline{I^a}$ for some rational number $a/b\geq t$, then $v(x^b)\geq v(I^a).$ 
        Since $R$ is Noetherian, $v(I^a)=av(I)$, see \cite[discussion after Definition 6.8.9]{HunekeSwansonIntegralClosure}. It follows that
        \[v(x)\geq\frac{a}{b}v(I)\geq t v(I)\]
        where $v$ runs over discrete rank-1 valuations of $K(R)$ over $R$.
       Thus $x\in I_t.$

       Since all Rees valuations are discrete of rank 1, we have that \autoref{lem.FractionalPowersNoetherianAlternateDescriptions.a} implies \autoref{lem.FractionalPowersNoetherianAlternateDescriptions.c}.

       Now suppose that $v(x) \geq t v(I)$  where $v$ runs over all Rees valuations $v_1,\ldots,v_n$ of $I$. {Consider a rational number $a/b$ such that $v_i(x) \geq \frac{a}{b} v_i(I)\geq t v_i(I)$  for all $i$}. Then \[v_i(x^b)=bv_i(x)\geq a v_i(I)=v_i(I^a).\]
       It follows that $x^b\in \overline{I^a}$.

       Analogously, \autoref{lem.FractionalPowersNoetherianAlternateDescriptions.f}, \autoref{lem.FractionalPowersNoetherianAlternateDescriptions.g} and \autoref{lem.FractionalPowersNoetherianAlternateDescriptions.h} are equivalent.}

        Let $I'_t$ denote the ideal made up of those elements satisfying condition \autoref{lem.FractionalPowersNoetherianAlternateDescriptions.d}.  Certainly $I'_t \subseteq I_t$ as $I'_t$ has more conditions on its elements.  Thus, take $x \in I_t$ and suppose that $v$ is a  valuation on $R$.  Write $t = a/b$ as in \autoref{eq.RationalPowerIdealPowers} so that $x^b \in \overline{I^a}$.  Thus $v(x^b) \geq v(I^a)$.  As $R$ is Noetherian, $v(I^a) = av(I)$.  Hence we have that 
        \[
            v(x) \geq {a \over b} v(I) = t v(I)
        \]
        as desired.


        For the final equivalence, notice that every Rees valuation is divisorial by \cite[Proposition 10.4.3]{HunekeSwansonIntegralClosure}.  As we already have the equivalence of \autoref{lem.FractionalPowersNoetherianAlternateDescriptions.a} and \autoref{lem.FractionalPowersNoetherianAlternateDescriptions.c}, and this condition lies between those two, we are done.
    \end{proof}

    We also make the following observation.

    \begin{lemma}
        \label{lem.IBiggerThanIsDefinedViaUnion}
        With notation as above $I_{>t} = \bigcup_{\epsilon > 0} I_{t+\epsilon}$ {where $\epsilon > 0$ runs over the real numbers such that $t+\epsilon$ is rational. Additionally, $I_{t} = \bigcup_{\epsilon \geq 0} I_{t+\epsilon}$ where $\epsilon > 0$ runs over the real numbers such that $t+\epsilon$ is rational.}
    \end{lemma} 
    \begin{proof}
        {In both cases the containment $\supseteq$  follows from the definitions.}
        Suppose $I = (f_1, \dots, f_n)$.  If $x \in I_{>t}$, then we can find some $\epsilon > 0$ {so that $t+\epsilon$ is rational} and $v(x) \geq (t+\epsilon) v(f_i)$ for the finitely many Rees valuations $v$.  But we need only check these finitely many Rees valuations by \cite[Proposition 10.5.2]{HunekeSwansonIntegralClosure} or \autoref{lem.FractionalPowersNoetherianAlternateDescriptions}.  The first result follows.

        { Likewise consider $x \in I_{t}$. Then, we can find some $\epsilon\geq0$ such that $t+\epsilon$ is rational and $v(x) \geq (t+\epsilon) v(f_i)$ for the finitely many Rees valuations $v$ and so $x \in I_{t+\epsilon}.$}
    \end{proof}

    In view of the above, we make the following definition.  We notice we only work with rational $t$ below.

    \begin{definition}
        \label{def.GeneralRationalIntegralClosure}
        Suppose $R$ is a (possibly non-Noetherian) integral domain, $I \subseteq R$ is an ideal, and $t \geq 0$ is a \emph{rational} number.  We define $I_t$ to be the set of $x \in R$ such that 
        for each valuation $v$ of $K(R)$ over $R$, we have that $v(x) \geq t v(y)$ for some $y \in I$.  It is straightforward to see this is an ideal.  Likewise, we define $I_{>t}$ to be the set of all $x$ such that $v(x) > tv(y)$ for some $y \in I$ for each valuation $v$ of $K(R)$.
    \end{definition}    

    \begin{remark}
        We certainly have $\bigcup_{\epsilon > 0} I_{t+\epsilon} \subseteq I_{> t}$ but we do not see equality in this generality.  {Although see \autoref{rem.UnionForR+} below where we note that these agree for the absolute integral closure of a Noetherian domain when $I$ is finitely generated.  This is because we can reduce the problem to the Noetherian case in this situation.}
    \end{remark}

    We will develop the non-Noetherian theory we need in  \autoref{subsec.RelationWithIntegralClosureAndFrationalPowers} below.


\subsection{Big Cohen-Macaulay $R^+$-algebras}

\begin{definition}[\BCM algebras]
    Suppose $(R, \fram)$ is an excellent Noetherian local domain.  Recall that a \emph{balanced Big Cohen-Macaulay algebra} is an $R$-algebra $B$ such that every system of parameters on $R$ forms a regular sequence on $B$.  We use \emph{BCM} as shorthand to denote balanced Big Cohen-Macaulay.
\end{definition}

BCM algebras exist in characteristics $p > 0$ and 0 by \cite{HochsterHunekeInfiniteIntegralExtensionsAndBigCM} (\cf \cite{HochsterHunekeApplicationsofBigCM}) and in mixed characteristic by \cite{AndreWeaklyFunctorialBigCM}.  In fact, in characteristic $p > 0$, the absolute integral closure of an excellent local domain $R$ in some $\overline{K(R)}$, denoted $R^+$, is \BCM by \cite{HochsterHunekeInfiniteIntegralExtensionsAndBigCM}.  In mixed characteristic \cite{BhattAbsoluteIntegralClosure,BMPSTWW-MMP} showed that the $p$-adic completion $\widehat{R^+}$, of $R^+$ is \BCM.  Using ultra products, there is a variant of $R^+$ that is \BCM in characteristic zero \cite{SchoutensCanonicalBCMAlgebrasAndRational} and \BCM $R^+$-algebras\footnote{that is, \BCM $R$-algebras that are also $R^+$-algebras} are known to exist thanks to \cite{MurayamaSymbolicTestIdeal}.  Finally, we note that if $B$ is a BCM $R$-algebra (or even an $R$-algebra such that a single system of parameters is a regular sequence, that is it is potentially unbalanced), then the $\fram$-adic completion is also a BCM $R$-algebra (that is, it is balanced), see for instance \cite{BartijinStrooker.ModificationsMonomiales} or \cite[Exercise 8.1.7 \& Theorem 8.5.1]{BrunsHerzog}.  

Suppose that $R$ is a complete local domain of characteristic $p > 0$.  Note, if $B$ is a BCM $R$-algebra, then $J^* \supseteq (JB) \cap R$ as $B$ is solid, see \cite[Theorem 8.6(b)]{HochsterSolidClosure}.  Interestingly, the reverse containment holds for sufficiently big $B$ as can be seen by combining \cite[Theorem 11.1]{HochsterSolidClosure} (or \cite{GabberMSRINotes}) with \cite[Theorem 8.4]{DietzBCMSeeds} (as any set of \BCM $R$-algebras admits a map to a larger \BCM $R$-algebra).  

\begin{definition}
    \label{def.BCMTight}
    With notation as above, still in characteristic $p > 0$, if $B$ is a \BCM $R$-algebra such that  $(JB) \cap R = J^*$ for all ideals $J \subseteq R$ we say that $B$ \emph{captures tight closure for $R$}.  In this case, we say that $B$ is a \BCMTight \emph{$R$-algebra}. 
\end{definition}

Note, if $B$ captures tight closure, so does any larger \BCM algebra, such as $B_{\perf} = \colim_{F} B = \bigcup_e B^{1/p^e}$.  In particular, we can assume that $B$ is perfect (and in particular has a map from $R_{\perf}$).

Even if $B$ is not \BCMTight, it still can produce a nice ``closure'' operation by extension and contraction.  

\begin{definition}
    Suppose $R$ is a ring and $B$ is an $R$-algebra (\BCM or not).  For any ideal $\fra \subseteq R$ we define $\fra^{\mathrm{cl}_B} := \fra B \cap R$.  In the special case that $B = R^+$ when $R$ is a domain, we define $\fra^+ := \fra R^+ \cap R$ (the \emph{plus closure}).
\end{definition}

It will be important for us that 
\begin{equation}
    \label{eq.PlusClosureInIntegralClosure}
    \fra^+ \subseteq \overline{\fra} 
\end{equation}
where the right side is the integral closure of $\fra$.
In characteristic $p > 0$ this follows as $\fra^* \subseteq \overline{\fra}$, \cite[(5.2) Theorem]{HochsterHunekeTC1}.  In characteristic zero this follows from the fact that $\fra^+ = \fra R^{\mathrm{N}} \cap R$ as any finite ring extension from a normal domain splits.
In mixed characteristic, see \cite[Proposition 2.6]{Heitmann.ExtensionsOfPlusClosure}.  

Suppose now we are given $S$ an $R$-algebra such that $R \subseteq S \subseteq R^+$.  The following is well known.

\begin{lemma}
    Suppose $R \subseteq S$ is a finite extension of complete local Noetherian domains.  An $S$-algebra $B$ is \BCM over $S$ if and only if it is \BCM over $R$.
\end{lemma}
\begin{proof}
    First we show nontriviality, let $\m_R$ and $\m_S$ be the maximal ideals of $R$ and $S$ respectively. There exists $n>0$ such that $\m_S^n\subseteq\m_RS$. Thus, 
    \[\m_S^n B\subseteq(\m_RS) B=\m_R B\subseteq \m_SB.\]
    and, since $\m_SB=B$ if and only if $\m_S^n B=B$,
    it follows that $\m_RB\neq B$ if and only if $\m_SB\neq B.$

    Now we show the depth condition. Any \BCM $S$-algebra is automatically a BCM $R$-algebra as any system of parameters of $R$ is also a system of parameters of $S$.
	To prove the converse, suppose that $B$ is \BCM over $R$. 
    Fix $\frq$ to be a prime of $\Spec S$ and let $\frp = \frq \cap R$.  
	Set $\frp\in \Spec(R)$ and let $\frq_1 := \frq,\frq_2,\ldots, \frq_n$ be the prime ideals of $S$ lying over
	$\frp$. 
	We have that
	$$0=H^i_{\frp R_\frp}(B_\frp)=H^i_{\frp S}(B S_\frp)= \prod H^i_{\frq_j}(B_\frp).$$
	Thus, $H^i_{\frq_j}(B_\frp)=0$ for each $j$ and so $H^i_{\frq_jS_{\frq_j}}(B_{\frq_j})=0$.
	By \cite[Corollary 2.8]{BhattAbsoluteIntegralClosure}, $B$ is BCM over $S$. 
\end{proof}

In characteristic $p > 0$, it is also possible to find a \BCM $R^+$-algebra that is \BCMTight for every finite $R$-algebra $S \subseteq R^+$.  We believe this is known to experts but we are not aware of a reference.  We use a construction of Gabber to accomplish this.

\begin{lemma}[\cf \cite{GabberMSRINotes}]
    \label{lem.BCMTightExist}
    Suppose $R$ is a complete local Noetherian domain of characteristic $p > 0$.  Then there is a \BCM $R^+$-algebra $B$ such that $B$ is \BCMTight for every finite $R$-algebra $S \subseteq R^+$.  We call such a $B$ a \emph{\BCMTight $R^+$-algebra}.
\end{lemma}
\begin{proof}
	Consider $T=\prod_\bN R^+$ with the diagonal map $R^+\rightarrow T$. For each finite extension  $R\subseteq S \subseteq R^+$, let $c_S\in R$ be a test element for $S$ and let $\mathbf{c}_S=(c_S,c_S^{1/p}, c_S^{1/p^2}, \ldots).$ Let $W\subseteq T$ be the multiplicative set generated by $\{\mathbf{c}_S  \;|\; R\subseteq S \subseteq R^+ \text{ a finite ring extension}\}$ and let $T'=W^{-1}T$.  First we prove that $T'$ is a \BCM algebra, which is an argument due to Gabber.  We include it for the convenience of the reader.
	
	Let $x_1,\ldots, x_n\in S$ be a system of parameters for $S$.
	Note that we have $(x_1,\ldots, x_n)T=\prod_\bN(x_1,\ldots, x_n)R^+$ since $(x_1,\ldots, x_n)$ is finitely generated.
	Now suppose that $t=(t_0,t_1,t_2, \ldots)\in T$ is such that $t\in\ker(T/(x_1,\ldots, x_{i-1})T\xrightarrow{\cdot x_i}T/(x_1,\ldots, x_{i-1})T).$	Then, for each $j$, $x_it_j\in (x_1,\ldots, x_{i-1})R^+=(x_1,\ldots, x_{i-1})S^+$ and, since $S^+$ is a BCM $S$-algebra, $t_j\in(x_1,\ldots, x_{i-1})S^+$. Thus, $t\in (x_1,\ldots, x_{i-1})T$.  Hence $x_i$ acts injectively on the localization $T'/(x_1,\ldots, x_{i-1})T'$ as well.
	
	On the other hand, suppose that $1\in \m T'$ where $\m=(y_1,\ldots, y_l)$ is the maximal ideal of $S$. Then, $1=\sum \mathbf{a}_j\frac{y_j}{\mathbf{c_j}}$. Clearing denominators, we have that there exist $\mathbf{c}=(c,c^{1/p}, c^{1/p^2}, \ldots)\in W$ such that $\mathbf{c}\in \m T.$ Equivalently, $c^{1/p^e}\in \m R^+$ for all $e$. Thus, for each $e$,
	$$c^{1/p^e}=\sum a_{e,j}y_j.$$
	Let $\nu$ be a valuation of $R^+$ that is an extension of a valuation centered at $\m$. Then,
	$$\frac{1}{p^e}\nu(c)=\nu(c^{1/p^e})\geq\min\{\nu(a_{e,j}y_j)\}=\min\{\nu(a_{e,j})+\nu(y_j)\}\geq \min\{\nu(y_j)\}$$
	for all $e$, which is impossible given that $\min\{\nu(y_j)\}>0$. It follows that $ \m T'\neq T'.$
	Therefore, $x_1,\ldots, x_n$ is a regular sequence on $T'$.
	Hence, $T'$ is a $\BCM$ S-algebra.
	
	Now let $J$ be an ideal of $S$. Suppose that $x\in J^*$. Then, $c_S^{1/p^e}x\in JR_\perf\subseteq JR^+$ for all $e$, so $\mathbf{c}_Sx\in JT$ and $x\in JT'$.  Hence, $J^*\subseteq JT'\cap S.$ 
    Since $S$ is a complete local ring of prime characteristic, we have that $J^*=JT'\cap S.$ Therefore, $T'$ is a \BCMTight S-algebra.
\end{proof}

{
\subsection{Test ideals in characteristic $p > 0$}

    We recall the definition and basic properties of (generalized) test ideals in characteristic $p > 0$.

    \begin{definition}[{\cite{HaraYoshidaGeneralizationOfTightClosure}, \cf \cite{HochsterFoundations,HochsterHunekeTC1,HaraInterpretation,HaraTakagiOnAGeneralizationOfTestIdeals}}]
        Suppose $(R,\fram)$ is an $F$-finite Noetherian local domain of characteristic $p > 0$, $\fra \subseteq R$ is an ideal, and $t > 0$ is a real number.  Then we define the (generalized) \emph{test ideal of $(R, \fra^t)$} to be 
        \[
            \tau(R, \fra^t) = \sum_{e > 0} \sum_{f \in \fra^{\lceil tp^e \rceil}} \Image\big((c f)^{1/p^e}\Hom_R(R^{1/p^e}, R) \to R\big).
        \]
        In the above, $c$ is a big test element for $R$ and the map is evaluation-at-1.  If $\fra = R$ or $t = 0$, then we simply denote it by $\tau(R)$.  Summing over $e \gg 0$ instead of $e > 0$ also produces the same ideal.

        Set $\omega_R$ to be a canonical module of $R$.  We define the \emph{test module of $(R, \fra^t)$} to be 
        \[
            \tau(\omega_R, \fra^t) = \sum_{e > 0} \sum_{f \in \fra^{\lceil tp^e \rceil}} \Image\big((c f)^{1/p^e}\Hom_R(R^{1/p^e}, \omega_R) \to \omega_R\big)
        \]
        with notation as above.  Again, we can also sum over $e \gg 0$, and if $\fra = R$ or $t = 0$, then we simply denote the test module by $\tau(\omega_R)$.
    \end{definition}

    Observe if $R$ is Gorenstein (or more generally $\omega_R \cong R$) then we can identify $\tau(R, \fra^t)$ with $\tau(\omega_R, \fra^t)$.  Note also that our test ideal is the \emph{big} test ideal (as opposed to the finitistic test ideal), which in some literature was denoted by $\tld{\tau}(R, \fra^t)$ or by $\tau_b(R, \fra^t)$.   We will be primarily interested in the regular and hence Gorenstein case, where the big and finitistic test ideals also coincide.  

    We summarize some basic properties of test ideals and test modules that we will use.

    \begin{proposition}
        \label{prop.BasicPropertiesOfTestIdeals}
        Suppose $R$ is an $F$-finite Noetherian local domain of characteristic $p > 0$, $\fra, \frb \subseteq R$ are proper nonzero ideals with $\fra \subseteq \sqrt{J}$, and $t > 0$ is a real number.  Fix $f \in R$.  We state these for $\tau(R, -)$, but everything below also holds for $\tau(\omega_R)$.
        \begin{enumerate}
            \item If $\fra \subseteq \frb$ is an ideal, then $\tau(R, \fra^t) \subseteq \tau(R, \frb^t)$ and if $0 \leq s < t$, then $\tau(R, \fra^t) \subseteq \tau(R, \fra^s)$.\label{prop.BasicPropertiesOfTestIdeals.EasyContainments}%
            \item $f \tau(R) = \tau(R, f^1)$.  \label{prop.BasicPropertiesOfTestIdeals.EasySkoda}
            \item $\tau(R, \fra^t) = \tau(R, \fra^{t+\epsilon})$ for all $1 \gg \epsilon > 0$.  \label{prop.BasicPropertiesOfTestIdeals.PlusEpsilonSmall}
            \item $\tau(R, \fra^t) = \tau(R, (\fra^m)^{t/m})$.  \label{prop.BasicPropertiesOfTestIdeals.Unambiguity}
            \item $\tau(R, \fra^t) = \tau(R, \overline{\fra}^t)$.\label{prop.BasicPropertiesOfTestIdeals.IntegralClosureAgnostic}
        \end{enumerate}
    \end{proposition}
    \begin{proof}
        Property \autoref{prop.BasicPropertiesOfTestIdeals.EasyContainments} follows directly from the definition.  
        \autoref{prop.BasicPropertiesOfTestIdeals.EasySkoda} is \cite[Proposition 1.11(1)]{HaraYoshidaGeneralizationOfTightClosure}.  Part \autoref{prop.BasicPropertiesOfTestIdeals.PlusEpsilonSmall} can be found in \cite[Lemma 3.23]{BlickleSchwedeTakagiZhang}, see \cite[Corollary 2.16]{BlickleMustataSmithDiscretenessAndRationalityOfFThresholds} in the regular case.  (d) follows from (c) and the definition as we have $\fra^{\lceil (t+\epsilon)p^e \rceil} \subseteq (\fra^m)^{\lceil tp^e/m \rceil} \subseteq \fra^{\lceil tp^e \rceil}$ for any $1 \gg \epsilon > 0$ and all $e \gg 0$.
         
        Property (e) is implicit in \cite{HaraYoshidaGeneralizationOfTightClosure}, or see \cite[Lemma 2.27]{BlickleMustataSmithDiscretenessAndRationalityOfFThresholds} in the regular case.  It also proceeds as follows.  Note $\subseteq$ is clear.   The reverse containment follows from (c) as we can find $m$ such that $\overline{\fra^{n+m}} \subseteq \fra^n$ for all $n$.
    \end{proof}
}

\subsection{\BCM test ideals}
    \label{subsec.BCMTestIdeals}
    We will limit ourselves primarily to the complete local case.  We begin with the historical definition.

    \begin{definition}[\cite{RobinsonToricBCM,MurayamaSymbolicTestIdeal}, \cf \cite{MaSchwedePerfectoidTestideal,MaSchwedePerfectoidTestideal,HaraYoshidaGeneralizationOfTightClosure}]
        \label{def.TestIdealsViaMatlisDuality}
        Fix $(R, \fram)$ to be a complete local Noetherian domain of dimension $d$ with canonical module $\omega_R$ and let $\fra \subseteq R$ be an ideal.  Fix $B$ to be a \BCM $R^+$-algebra.  For $ t > 0$, we define the \BCM test module $\tau_B(\omega_R, \fra^t)$ to be
        \[
            \sum_{n>0} \sum_{f \in \fra^{\lceil tn \rceil}}\Big(\Image\big( H^d_{\fram}(R) \xrightarrow{f^{1/n}} H^d_{\fram}(B)\big)\Big)^{\vee}
        \]
         viewed as a submodule of $\omega_R$  
        where $-^{\vee}$ denotes Matlis dual $\Hom_R(-, E)$ and $E$ is an injective hull of $k = R/\fram$.  Note $f^{1/n} \in R^+ \subseteq B$ and the particular choice of root only differs by a unit in $R^+$ which does not impact the image.

        Suppose $R$ is additionally $\bQ$-Gorenstein.  We define $\tau_B(R, \fra^t)$ to be  
        \[
            \sum_{n>0} \sum_{f \in \fra^{\lceil tn \rceil}} \Big(\Image\big( H^d_{\fram}(\omega_R) \xrightarrow{f^{1/n}} H^d_{\fram}(B \otimes_R \omega_R)\big)\Big)^{\vee}.
        \]
        viewed as an ideal of $R \cong H^d_{\fram}(\omega_R)^\vee$.
    \end{definition}
    The definition of $\tau_B(R)$ makes sense without the $\bQ$-Gorenstein condition, but then there are other potential definitions that we do not know coincide (and there are even some issues in characteristic $p > 0$).  For some discussion of how these definitions compare in the case where $\tau_B(R) = R$ (the \BCM-regular case) see \cite[Definition 5.3.1]{CaiLeeMaSchwedeTuckerPerfectoidHilbertKunz}.

    We now state Matlis dual formulations of the above definitions.  These statements are well known to experts, but we include them for future reference of the reader.

    \begin{lemma}
        \label{lem.UnDualDefinitionOfTestIdeal}
        With notation as in \autoref{def.TestIdealsViaMatlisDuality} we have that 
        \[
            \tau_B(\omega_R, \fra^t) = \sum_{n>0} \sum_{f \in \fra^{\lceil tn \rceil}} \Image\big( f^{1/n}\Hom_R(B, \omega_R) \to \omega_R\big)
        \]
        and, in the case that $R$ is $\bQ$-Gorenstein, that 
        \[
            \tau_B(R, \fra^t) = \sum_{n>0} \sum_{f \in \fra^{\lceil tn \rceil}} \Image\big( f^{1/n}\Hom_R(B, R) \to R\big)
        \]
    \end{lemma}
    \begin{proof}
        For the first statement, it suffices to show that the Matlis dual of
        \[
            \Image\big( H^d_{\fram}(R) \xrightarrow{f^{1/n}} H^d_{\fram}(B))
        \]
        is the same as the image of the evaluation-at-$f^{1/n}$-map $\Hom_R(B, \omega_R) \to \omega_R$. 

        Since $R$ is complete, we know that $\Hom_R(H^d_{\fram}(R), E) \cong \omega_R$.  With $E$ as above, using adjointness,
        \[            
            \begin{array}{rcl}
            H^d_{\fram}(B)^{\vee} & = & \Hom_R(H^d_{\fram}(B), E) \\
            & \cong & \Hom_R(B \otimes H^d_{\fram}(R), E)\\
            & \cong & \Hom_R(B, \Hom_R(H^d_{\fram}(R), E))\\
            & \cong & \Hom_R(B, \omega_R).
            \end{array}
        \]
        As $E$ is injective, $(-)^{\vee} = \Hom_R(-, E)$ is exact, and we have a factorization
        \[
            \omega_R \cong H^d_{\fram}(R)^{\vee} \hookleftarrow \Image\big( H^d_{\fram}(R) \xrightarrow{f^{1/n}} H^d_{\fram}(B)\big)^{\vee} \twoheadleftarrow H^d_{\fram}(B)^{\vee} \cong \Hom_R(B, \omega_R).
        \]
        Tracing through the isomorphisms, this proves exactly what we wanted.

        
        For the second case, it suffices to show that 
        \[
            \Image\big( H^d_{\fram}(\omega_R) \xrightarrow{f^{1/n}} H^d_{\fram}(B \otimes \omega_R))
        \]
        is the same as the image of the evaluation-at-$f^{1/n}$-map $\Hom_R(B, R) \to R$.  The argument is the same as above except we notice $H^d_{\fram}(\omega_R) \cong E$, $\Hom_R(H^d_{\fram}(\omega_R), E) \cong R$, and that 
        \[ 
            \Hom_R(H^d_{\fram}(B \otimes \omega_R), E) \cong \Hom_R(B, \Hom_R(H^d_{\fram}(\omega_R), E)) \cong \Hom_R(B, R).
        \]        
    \end{proof}    

    We will find the following results useful.  
    \begin{theorem}[{\cite[Proposition 8.10]{BMPSTWW-RH}}]
        Suppose $(R, \fram)$ is a complete local domain of residue characteristic $p > 0$.  Fix $\fra \subseteq R$ an ideal and $t > 0$.  Set $B = \widehat{R^+}$ to be the $p$-adic completion of $R^+$ (which is \BCM by \cite{HochsterHunekeInfiniteIntegralExtensionsAndBigCM,BhattAbsoluteIntegralClosure}).  Set $\pi : X \to \Spec R$ to be a proper birational map, with $X$ normal, factoring through the blowup of $\fra$.  Fix $\fra \cO_X = \cO_X(-G)$.  Then for all $1 \gg \epsilon > 0$ we have that
        \[
            \tau_B(\omega_R, \fra^{t+\epsilon}) = \Image\Big( H^d_{\fram}(R) \to H^d_{\fram}(\myR \Gamma(X^+, \cO_{X^+}((t+\epsilon)G)))\Big)^{\vee}.
        \]        
        Here by $\cO_{X^+}((t+\epsilon)G)$ we mean the pullback to $X^+$ of $\cO_Y(h^*(t+\epsilon) G)$ for any finite map $h : Y \to X$ with $Y$ integral such that $h^*(t+\epsilon)G$ has integer coefficients.
    \end{theorem}

    \begin{theorem}[{\cite[Theorem 8.11]{BMPSTWW-RH}}]
        \label{thm.MixedCharBlowupComparison}
        Suppose $(R, \fram)$ is a complete regular local domain of mixed characteristic $(0,p > 0)$.  Fix $\fra \subseteq R$ an ideal and $t > 0$.  Set $B$ to be a perfectoid \BCM $R^+$-algebra.  Set $\pi : X \to \Spec R$ to be a proper birational map with $X$ normal factoring through the blowup of $\fra$.  Fix $\fra \cO_X = \cO_X(-G)$.  
        Then for all $ 1 \gg \epsilon > 0$ we have that 
        \[
            \tau_B(R, p^{\epsilon}\fra^{t+\epsilon}) = \Image\Big( H^d_{\fram}(R) \to H^d_{\fram}(\myR \Gamma(X^+, \cO_{X^+}((t+\epsilon)G + \epsilon \Div_X p)))\Big)^{\vee}.
        \]
    \end{theorem}
    \begin{proof}
        The result found in \cite[Theorem 8.11]{BMPSTWW-RH} has a constant which we may take as $c = p$ as $R$ is regular.  
    \end{proof}
    
    We conclude with a definition of jumping numbers of \BCM test ideals.

    \begin{definition}
        With notation as above and in particular $\bQ$-Gorenstein $R$, we define the $t > 0$ to be a \emph{$\BCM_B$-jumping number of $\fra \subseteq R$} if $\tau_B(R, \fra^{t-\epsilon}) \neq \tau_B(R, \fra^{t+\epsilon})$ for $1 \gg \epsilon > 0$.  
    \end{definition}
    Note, in characteristic $p > 0$, then this agrees with the usual notion of $F$-jumping number by a small modification of \cite[Corollary 6.23]{MaSchwedeSingularitiesMixedCharBCM}.  In mixed characteristic, if $B$ is perfectoid and sufficiently large, we have that $\tau_B(R, \fra^{t+\epsilon}) = \tau_B(R, \fra^{t})$ for all $1 \gg \epsilon > 0$ by \cite{BMPSTWW-RH}.  Hence, in either case, $t$ is a $\BCM_B$-jumping number of $\fra \subseteq R$ if and only if $\tau_B(R, \fra^{t-\epsilon}) \neq \tau_B(R, \fra^{t})$.

\section{$F$-thresholds via perfection and big Cohen-Macaulay algebras}

Our goal in this section is to explore the notion of the $F$-threshold in the style of \cite{RodriguezBCMThresholdsHypersurfaces}.  

We first pass from iterated Frobenius to the perfection.

\begin{proposition}
    Suppose $R$ is a Noetherian domain  of characteristic $p > 0$ and $J, \fra \subseteq R$ are proper nonzero ideals with $\fra \subseteq \sqrt{J}$.  Then 
    \[
        \tld c^J(\fra) = \sup\Big\{{n \over p^e} > 0 \;\Big|\; (\fra^{n})^{1/p^e} \nsubseteq JR_{\perf}\Big\}. 
    \]        
\end{proposition}
    Note first that $(\fra^n)^{1/p^e} \subseteq (\fra^{np^d})^{1/p^{e+d}}$ and so if ${n \over p^e}$ is in the set we are taking the supremum of, so is ${np^d \over p^{e+d}}$, which removes some worry about redundancy.  It is a variation on this observation that makes the existence of $\tld c^J(\fra)$ particularly easy to prove as the $\nu_{\fra}(p^e)/p^e$ are ascending \cite{HunekeMustataTakagiWatanabeFThresholdsTightClosureIntClosureMultBounds}.
\begin{proof}    
    By change of notation, $\fra^n \nsubseteq (J^{[p^e]})^F$ if and only if $(\fra^n)^{1/p^e} \nsubseteq (JR^{1/p^e})^F$ which occurs if and only if $(\fra^n)^{1/p^e} \nsubseteq JR_{\perf}$ by the definition of Frobenius closure.  In other words, $\tld \nu_{\fra}^J(p^e) = \sup\{n \;|\; (\fra^n)^{1/p^e} \nsubseteq JR_{\perf}\}$.  The result follows.
\end{proof}

\begin{proposition}
    Suppose $R$ is a complete local Noetherian domain of characteristic $p > 0$ and $B$ is a perfect \BCMTight $R$-algebra, then with $\fra$ and $J$ as above, 
    \[
        c_*^J(\fra) = \sup\Big\{ {n \over p^e} > 0 \;|\; (\fra^n)^{1/p^e} \nsubseteq JB \}.
    \]
\end{proposition}
\begin{proof}
    Note $\fra^n \nsubseteq (J^{[p^e]})^*$ if and only if $(\fra^n)^{1/p^e} \nsubseteq (J R^{1/p^e})^* = J B \cap R^{1/p^e}$ which happens if and only if $(\fra^n)^{1/p^e} \nsubseteq JB$.  This finishes the proof.
\end{proof}

When we generalize this to non-$p$-characteristic, we will not have a notion of taking $p$th roots or taking fractional powers of ideals in general.  Fortunately, as mentioned in the introduction, there already is a way to take a fractional power of an ideal up to integral closure.  Thus, for us, it is important to understand how integral closure behaves in this context.

\begin{proposition}
    \label{prop.IntegralClosureStillUsingPePowers}
    Suppose $R$ is a complete local Noetherian domain of characteristic $p > 0$ and $\fra, J \subseteq R$ are as above.  Then 
    \[
        \tld c^J(\fra) = \sup\Big\{ {n \over p^e} > 0 \;\Big|\; (\overline{\fra^n})^{1/p^e} \nsubseteq JR_{\perf} \Big\}
    \]
    and if $B$ is a \BCMTight $R$-algebra, then 
    \[
        c_*^J(\fra) = \sup\Big\{ {n \over p^e} > 0 \;\Big|\; (\overline{\fra^n})^{1/p^e} \nsubseteq JB \Big\}.
    \]    
\end{proposition}
\begin{proof}
    Fix a constant $k$ such that 
    \[
        \fra^{n+k} \subseteq \overline{\fra^{n+k}} \subseteq {\fra^n} 
    \]
    for all $n \geq 0$ (in fact, by \cite[Theorem 4.13]{HunekeUniformBoundsInNoetherian}, such a constant even exists independently of $\fra$, but we will not need that).   
    It follows that 
    \[
        \Big\{ {n \over p^e} > 0 \;\Big|\; ({\fra^n})^{1/p^e} \nsubseteq JR_{\perf} \Big\} 
        \supseteq 
        \Big\{ {{n - k} \over p^e} > 0 \;\Big|\; (\overline{\fra^{n}})^{1/p^e} \nsubseteq JR_{\perf} \Big\}
        \supseteq
        \Big\{ {n - k \over p^e} > 0 \;\Big|\; ({\fra^{n}})^{1/p^e} \nsubseteq JR_{\perf} \Big\}.
    \]
    As $({\fra^j})^{1/p^e} \subseteq ({\fra^{p^d j}})^{1/p^{e+d}}$, the supremum of the leftmost term is computed by $n,p^e \gg 0$.  As $k$ is constant, the suprema of the left and right sets then coincide, and coincide with the middle supremum.  Similarly, as $(\overline{\fra^j})^{1/p^e} \subseteq (\overline{\fra^{p^d j}})^{1/p^{e+d}}$, the middle supremum likewise also equals   
    \[
        \sup \Big\{ {{n} \over p^e} > 0 \;\Big|\; (\overline{\fra^{n}})^{1/p^e} \nsubseteq JR_{\perf} \Big\}
    \]
    This proves the first case.  The same argument proves the second statement by replacing $R_{\perf}$ with $B$.      
\end{proof}

    We now consider what happens for finite extensions.  Recall by \cite[Proposition 2.2(a)]{HunekeMustataTakagiWatanabeFThresholdsTightClosureIntClosureMultBounds} that if $R \subseteq S$ is pure, for instance if it splits, then $c^J(\fra) = c^{JS}(\fra S)$.  This result also holds for the tight closure $F$-threshold.
\begin{lemma}
    \label{lem.ThresholdFiniteIntegralExtension}
    Fix $R$ a complete local Noetherian domain of characteristic $p > 0$ with a finite extension $R \subseteq S \subseteq R^+$.  Set $B$ to be a perfect Cohen-Macaulay $R^+$-algebra that captures both $R$ and $S$ tight closure (see \autoref{lem.BCMTightExist}).  Then 
    \[
        c_*^J(\fra) = c_{*}^{JS}(\fra S).
    \]
    In particular, 
    \[
        c_*^J(\fra) = \sup\Big\{ {n \over p^e} > 0 \;\Big|\; (\overline{\fra^n S})^{1/p^e} \nsubseteq JB \Big\}.
    \]
\end{lemma}
\begin{proof}
    We notice that $(\fra^i)^{1/p^e} \subseteq JB$ if and only if $((\fra S)^i)^{1/p^e} \subseteq JB$ (in fact, if and only if $((\fra B)^i)^{1/p^e} \subseteq JB$), hence
    \[
        c_*^J(\fra) = \sup\Big\{ {n \over p^e} \;\Big|\; (\fra^n S)^{1/p^e} \nsubseteq JB \Big\} = c_*^{JS}(\fra S)
    \]
    proving the first statement.
    This agrees with $\sup\Big\{ {n \over p^e} > 0 \;\Big|\; (\overline{\fra^n S})^{1/p^e} \nsubseteq JB \Big\}$ by \autoref{prop.IntegralClosureStillUsingPePowers}.  
\end{proof}

\subsection{An interpretation via integral closure and fractional powers}
\label{subsec.RelationWithIntegralClosureAndFrationalPowers}

See \autoref{subsec.FractionalIntegrallyClosedPowers} for a brief introduction to the notion of fractional integrally closed powers.

Using our characterizations of the $F$-threshold above, we obtain the following.
\begin{theorem}
    \label{thm.FThresholdTheoremViaR+Theorem}
    Suppose $R$ is an $F$-finite Noetherian complete local domain of characteristic $p > 0$.  Set $B$ to be a perfect \BCMTight $R^+$-algebra.  
    Then for $\fra, J \subseteq R$ as above, we obtain that:
    \[
        c_*^J(\fra) = \sup \Big\{ t > 0 \;\Big|\; (\fra S)_{t} \nsubseteq JB \text{ for some finite $R \subseteq S \subseteq R^+$}\Big\}.
    \]
\end{theorem}
\begin{proof}
    Set $T$ to be the set we are taking the supremum of in the statement of the theorem.  
    For a fixed, $R \subseteq S \subseteq R^+$,  note that $\overline{((\fra S)^n)}^{1/p^e}=\overline{((\fra^{1/p^e} S^{1/p^e})^n)}$ since $S$ is abstractly isomorphic to $S^{1/p^e}$ via the $p$th root function.  On the other hand, for any valuation $\nu$ of $K(S^{1/p^e})$ over $S^{1/p^e}$, $\nu((\fra^{1/p^e} S^{1/p^e})^n)=\frac{n}{p^e}\nu(\fra S^{1/p^e})$ and so $\overline{((\fra S)^n)}^{1/p^e}=(\fra S^{1/p^e})_{n/p^e}$ again using the abstract isomorphism $S \cong S^{1/p^e}$.
    
    It immediately follows from \autoref{lem.ThresholdFiniteIntegralExtension} that  
    \[
        c_*^J(\fra) = \sup \Big\{ {n \over p^e} >0 \;\Big| (\fra S^{1/p^e})_{n/p^e} \nsubseteq JB \Big\}.
    \]
    Note, if $(\fra S^{1/p^e})_{t} \nsubseteq JB$, then we also have that the even larger set $(\fra S^{1/p^{e'}})_{t} \nsubseteq JB$ for $e' \geq e$.  
    Combining this with the fact that we can write any ${n / p^e} = {np^{e'-e} / p^{e'}}$ we see that 
    \[
        c_*^J(\fra) = \sup \Big\{ {n \over p^e}>0 \;\Big| (\fra S^{1/p^{e'}})_{n/p^e} \nsubseteq JB \;\text{ for some $e'$}\Big\}.
    \]    
    Using again that $(\fra S^{1/p^e})_{t} \subseteq (\fra S^{1/p^e})_{t'}$ for $t' \leq t$, and the fact that the numbers of the form $n/p^e$ are dense, we obtain that
    \[
        c_*^J(\fra) = \sup \Big\{ {t > 0} \;\Big| (\fra S^{1/p^{e'}})_{t} \nsubseteq JB \;\text{ for some $e' \geq 0$}\Big\}.
    \]
    Set $T_S$ to be the set appearing on the right in the above equality.  Each $T_S$ has the same supremum ($c_*^J(\fra)$) and so each is an interval whose interior is $(0, c_*^J(\fra))$.

    Suppose $t \in T$, then for some $S$, $t \in T_S$ (setting $e' = 0$).  Conversely, if $t \in T_S$ for some $S$, then as $R \subseteq S^{1/p^{e'}}$ is finite, we see that $t \in T$.  Hence $T = \bigcup_{S} T_S$.  But each $T_S$ has the same supremum, hence $c_*^J(\fra) = \sup T$ as desired.
\end{proof}

The characterization of $c_*^J(\fra)$ above makes sense in any characteristic, and indeed essentially will be our eventual definition (which avoids the intermediate rings $S$).  Towards that end consider the following lemma.

\begin{lemma}
    \label{lem.fractionalIntegralClosureInR+ViaNoetherian}
    Suppose $R$ is a Noetherian domain with $R^+ \subseteq \overline{K(R)}$ an absolute integral closure.  Then for any ideal $\fra = (f_1, \dots, f_n) \subseteq R$ and any $t \in \bQ_{>0}$ we have that 
    \[
        (\fra R^+)_t = \bigcup_{S} (\fra S)_t
    \]
    where the union runs over finite extensions $R \subseteq S \subseteq R^+$.  Likewise 
    \[
        (\fra  R^+)_{>t} = \bigcup_{S} (\fra S)_{>t}.
    \]
\end{lemma}
\begin{proof}
    If $x \in (\fra S)_t$ then for each valuation $v$ of $K(S)$ over $S$, we have that $v(x) \geq t v(f_i)$ for some $i = 1, \dots, n$.  For each valuation $v'$ of $\overline{K(S)} = K(R^+)$ over $R^+$, as it restricts to a valuation on $K(S)$ over $S$, we immediately see that $x \in (\fra  R^+)_t$.  Hence we have the containment $(\supseteq)$.  

    Conversely, if $x \in (\fra R^+)_t$, then for any $S$ containing $x$, and each valuation $v$ of $K(S)$ over $S$, it extends to a valuation $v'$ of $K(R^+)$.  We claim it is also non-negative on $R^+$. Indeed, the ring associated to $v'$ contains $S$, and is integrally closed, and hence must contain $R^+$.  We then have that $v'(x) \geq tv'(f_i)$ for some $i$, and so the same holds for $v = v'|_{K(S)}$.

    The second statement follows similarly simply replacing each $\geq$ with $>$.
\end{proof}

It immediately follows that \autoref{thm.FThresholdTheoremViaR+Theorem} can be restated as follows:  For $R$ an $F$-finite complete Noetherian local domain of characteristic $p > 0$ and for $B$ a \BCMTight $R^+$-algebra, we have that:
\begin{equation}
    \label{eq.ValuativeDescriptionOfcStarJ}
    c_*^J(\fra) = \sup \Big\{ t \in \bQ_{>0} \; |\; (\fra R^+)_t \nsubseteq JB \}
\end{equation}

{
\begin{remark}
    \label{rem.UnionForR+}
    In the setting of \autoref{lem.fractionalIntegralClosureInR+ViaNoetherian} we can deduce that $(\fra R^+)_{>t} = \bigcup_{\epsilon > 0} (\fra R^+)_{t+\epsilon}$ where $\epsilon > 0$ runs over the real numbers such that $t+\epsilon$ is rational.  As the $\supseteq$ containment is clear, suppose that $x \in (\fra R^+)_{> t}$.  Then $x \in (\fra S)_{> t}$ for some finite extension $R \subseteq S$.  As $S$ is Noetherian, if we let $v_1, \dots, v_n$ denote the Rees valuations of $\fra S$, then $v_i(x) > t v_i(\fra S)$ for all $i$.  Hence there exists some $\epsilon > 0$ such that $t+\epsilon$ is rational and $v_i(x) \geq (t+ \epsilon) v_i(\fra S)$ for all $i$ and so $x \in (\fra S)_{>t+\epsilon}$.  But then $x \in (\fra R^+)_{>t+\epsilon}$ as well.   
    Similarly, we have that $(\fra R^+)_{t} = \bigcup_{\epsilon \geq 0} (\fra R^+)_{t+\epsilon}$
\end{remark}

In view of the previous remark we make the following definition which generalizes \autoref{def.GeneralRationalIntegralClosure} from the case of rational numbers to real numbers

\begin{definition}
    \label{def.FractionalIntegralClosureForRealInR+}
    Suppose $R$ is a Noetherian domain with $R^+ \subseteq \overline{K(R)}$ an absolute integral closure.  Then for any ideal $\fra \subseteq R$ and any {\emph{real}} $t>0$ we define
    $$(\fra R^+)_{>t} = \bigcup_{\epsilon > 0} (\fra R^+)_{t+\epsilon}$$
    where $\epsilon > 0$ runs over the real numbers such that $t+\epsilon$ is rational.
    Similarly, we define, 
    $$(\fra R^+)_{t} = \bigcup_{\epsilon \geq 0} (\fra R^+)_{t+\epsilon}$$
    where $\epsilon \geq 0$ runs over the real numbers such that $t+\epsilon$ is rational.
\end{definition}
}

We finally obtain the following.

\begin{corollary}
    \label{cor.FinalDescriptionOfcstar}
    Suppose $R$ is an $F$-finite complete local domain of characteristic $p > 0$, $B$ is a perfect \BCMTight $R^+$ -algebra, and $\fra, J \subseteq R$ are proper nonzero ideals with $\fra \subseteq \sqrt{J}$.  Then 
    \begin{equation}
        c_*^J(\fra) = \sup \Big\{ t \in \bR_{\geq 0} \;|\; (\fra R^+)_{> t} \nsubseteq JB \}.
    \end{equation}    
\end{corollary}
Note we take $t \geq 0$ (instead of $t > 0)$ since we are considering $(\fra R^+)_{>t}$.
\begin{proof}
    As $ (\fra R^+)_{> t} \subseteq  (\fra R^+)_{t}$, we observe that 
    \[
        \Gamma' := \Big\{ t \in \bQ_{\geq 0} \;|\; (\fra R^+)_{> t} \nsubseteq JB \} \subseteq \Big\{ t \in \bQ_{\geq 0} \; |\; (\fra R^+)_t \nsubseteq JB \} =: \Gamma.
    \] 
    Hence the supremum of $\Gamma'$ is less than or equal to that of $\Gamma$ (note whether we require $\geq 0$ or $> 0$ in the definitions of $\Gamma$ or $\Gamma'$ does not impact the supremum).  Thus, we see that $c_*^J(\fra) = \sup \Gamma \geq \sup \Gamma'$.

    Now, suppose that $s \in \Gamma$ and hence also we have $s-\epsilon \in \Gamma$ for all $1 \gg \epsilon > 0$ {with $t+\epsilon$ rational}.  But then as $(\fra R^+)_{s} \subseteq (\fra R^+)_{> s - \epsilon}$ we see that $s - \epsilon \in \Gamma'$.  Thus $\sup \Gamma' \geq s$, and hence $\sup \Gamma' \geq \sup \Gamma$. {Given that $\bQ$ is dense in $\bR$, the result follows.}
\end{proof}


We conclude this section with several additional lemmas that we will need later.



\begin{lemma}
    \label{lem.IntegralClosureExtensionFractionalPowersExtensionContraction}
    Suppose $R \subseteq S \subseteq R^+$ is a finite extension of Noetherian domains.  Suppose $I \subseteq R$ is an ideal.  Then $(IS)_{\alpha} \cap R = I_{\alpha}$ and hence $(IS)_{> \alpha} \cap R = I_{> \alpha}$ for all $\alpha>0$.
\end{lemma}
\begin{proof}
    First suppose that $\alpha$ is a rational number.
    Write $I = (f_1, \dots, f_n)$.  Any valuation $v'$ over $S$ of $K(S)$ restricts to a valuation $v$ over $R$ on $K(R)$, and conversely any valuation of $K(R)$ extends to at least one of $K(S)$ by a corollary to Chevalley's Extension Theorem, see for instance \cite[Theorem 3.1.2]{EnglerPrestel.ValuedFields}.
    Therefore, the elements $y$ of $R \subseteq S$ such that $v'(y) \geq \alpha v'(I) = \alpha \min\{v'(f_i)\} = \alpha \min\{v(f_i)\} = \alpha v(I)$ for all $v'$ are exactly those that satisfy the same condition for all $v$.  
    The result now follows from \autoref{lem.IBiggerThanIsDefinedViaUnion}.
\end{proof}

\begin{lemma}
    \label{lem.FractionalPowerOfIdealPowersAndProducts}
    Suppose $R$ is a Noetherian domain with $R^+ \subseteq \overline{K(R)}$ an absolute integral closure.  Suppose further that $\fra \subseteq R$ is an ideal.  
    Then $(\fra^n R^+)_{t} = (\fra R^+)_{nt}$ and $(\fra^n R^+)_{>t} = (\fra R^+)_{>nt}$ for {real} $t > 0$ and any integer $n \geq 1$.

    Furthermore, for any {real }$s,t > 0$, we have that $(\fra R^+)_{s} \cdot (\fra R^+)_{t} \subseteq (\fra R^+)_{s+t}$.  
\end{lemma}
\begin{proof}
    { By definition, we can assume that $s$ and $t$ are rational.}
    Write $\fra = (f_1, \dots, f_n)$.  If $x \in (\fra^n R^+)_{t}$, then for each $v$ of $K(R^+)$ over $R^+$, we have that $v(x) \geq t v(\prod_j f_j^{a_j} ) = t \sum_j a_j v(f_j)$ for some integers $a_j$ with $\sum_j a_j = n$.  But then $v(x) \geq tn v(f_i)$ for some single $i$ (corresponding to the smallest $v(f_i)$) and so $x \in (\fra R^+)_{nt}$.  This proves that $(\fra^n R^+)_t \subseteq (\fra R^+)_{nt}$.
    
    Conversely, if $y \in (\fra R^+)_{nt}$ then for each $v$ as above, $v(y) \geq nt v(f_i)$ for some $i$ and so $v(y) \geq t v(f_i^n)$.  But $f_i^n \in \fra^n R^+$ and so $y \in (\fra^n R^+)_t$.  This shows that $(\fra R^+)_{nt} \subseteq (\fra^n R^+)_{t}$

    The statement with $> t$ is the same, simply replace $\geq$ with $>$.  

    For the final statement, \autoref{lem.fractionalIntegralClosureInR+ViaNoetherian} lets us reduce to the Noetherian case which is \cite[Proposition 10.5.2(3)]{HunekeSwansonIntegralClosure}.
\end{proof}

\section{\BCM-thresholds, bounds and parameter ideals}

The goal of this section is to define the general \BCM-threshold and then translate some of the results of \cite{HunekeMustataTakagiWatanabeFThresholdsTightClosureIntClosureMultBounds} for $F$-thresholds and parameter ideals into a characteristic free environment.

\begin{definition}
    \label{def.BCMThresholdGeneral}
    Suppose $(R, \fram)$ is a Noetherian local domain and $B$ is a \BCM $R^+$-algebra.  Fix $\fra, J \subseteq R$ proper nonzero ideals of $R$ with $\fra \subseteq \sqrt{J}$.  We define the \emph{\BCM-threshold of $\fra$ with respect to $J$ along $B$} to be:
    \[
        c_B^J(\fra) := \sup \{t \in \bR_{\geq 0} \;|\; (\fra R^+)_{>t} \not\subseteq JB \} = \inf\{t \in \bR_{\geq 0} \;|\; (\fra R^+)_{>t} \subseteq JB \}
    \]
\end{definition}

In fact, if $c = c^J_B(\fra)$, then as $(\fra R^+)_{>c + \epsilon} \subseteq JB$ for all $\epsilon > 0$ such that $c+\epsilon$ is rational, we see from \autoref{def.FractionalIntegralClosureForRealInR+} that $(\fra R^+)_{>c}  \subseteq JB$.  Hence, we obtain the convenient fact that 
\begin{equation}
    \label{eq.BCMThresholdIsAMin}
    c_B^J(\fra) = \min\{t \in \bR_{\geq 0} \;|\; (\fra R^+)_{>t} \subseteq JB \}.
\end{equation}

We make some small observations about \BCM-thresholds.
\begin{lemma}
    With notation as in \autoref{def.BCMThresholdGeneral}, we have the following.
    \begin{enumerate}
        \item If $\fra \subseteq \frb$ is an ideal then $c_B^J(\fra) \leq c_B^J(\frb)$.
        \item If $J \subseteq I$ is an ideal then $c_B^I(\fra) \leq c_B^J(\fra)$.
        \item We have that $c_B^J(\fra) = c_B^{J^{\cl_B}}(\fra)$ where $J^{\cl_B} := JB \cap R$.
        \item We have that $c_B^J(\fra) = c_B^J(\overline{\fra})$.
        \item We have that $c_B^J(\fra^n) = {1 \over n} c_B^J(\fra)$ for each integer $n \geq 1$.
    \end{enumerate}
\end{lemma}
\begin{proof}
    For the first statement, observe that if $\fra \subseteq \frb$, then $(\fra R^+)_{>t} \subseteq (\frb R^+)_{>t}$, and the statement follows.  For the next statement, simply notice that $IB \supseteq JB$.  The third statement follows as $J^{\cl_B} B = J B$.  
    
    For the next statement we wish to argue that $(\fra R^+)_{>t} = (\overline{\fra} R^+)_{>t}$.  It suffices to show that if $x \in \overline{\fra}$ and $\frb = \fra + (x)$, then $(\fra R^+)_{>t} = (\frb R^+)_{>t}$.  But notice that if $\fra = (g_1, \dots, g_n)$, then for each valuation $v$ of $K(R)$ (or equivalently of $K(R^+)$) we have that $v(x) \geq v(g_i)$ for some $i$.  Thus if $y \in (\frb R^+)_{>t}$ and if $v(y) > t v(x)$, we also have that $v(y) > t v(g_i)$ and hence $y \in (\fra R^+)_{>t}$ as desired.

    For the last statement, apply \autoref{lem.FractionalPowerOfIdealPowersAndProducts}.
\end{proof}

To show some basic bounds on our \BCM-threshold, we will need an additional assumption on $B$.

\begin{definition}[{\cf \cite[Axiom (9)]{MurayamaSymbolicTestIdeal}}]
    \label{def.BrianconSkodaProperty}
    Suppose $(R, \fram)$ is a complete local Noetherian domain. 
    We say that a \BCM $R^+$-algebra $B$ satisfies the \emph{Brian{\c{c}}on-Skoda-property} if for each  ideal $J = (f_1, \dots, f_l) \subseteq R$ and each finite extension $R \subseteq S \subseteq R^+$, we have that 
    \[
        \overline{J^{n+l-1}S} \subseteq J^n B
    \]
    for every $n> 0$.
\end{definition}

The most common \BCM $R^+$-algebras satisfy this property.  Indeed, any \BCM algebra satisfying sufficiently good weak functoriality satisfies this property by \cite{RodriguezSchwedeBrianconSkodaViaWeakFunctoriality}.  In particular, $R^+$ satisfies this in characteristic $p > 0$ (\cite{HochsterHunekeApplicationsofBigCM}), the $p$-adic completion $\widehat{R^+}$ satisfies this in mixed characteristic, and such $R^+$-algebras can be constructed in characteristic zero as in \cite[Section 2.5]{MurayamaSymbolicTestIdeal}, see also \cite{HochsterHunekeApplicationsofBigCM,SchoutensCanonicalBCMAlgebrasAndRational,AschenbrennerSchoutensLefschetzExtensionsTightClosureAndBCM,DietzBCMSeeds,  GabberMSRINotes, AndreWeaklyFunctorialBigCM}.

This condition forces our \BCM-threshold to be finite if $\fra \subseteq \sqrt{J}$.

\begin{proposition}
    \label{prop.FinitenessAssumingBS}
    With notation as in \autoref{def.BCMThresholdGeneral}, suppose $\fra \subseteq \sqrt{J}$ and that $B$ satisfies the Brian{\c{c}}on-Skoda property.  Then $c_B^J(\fra) < \infty$.  More specifically, if $\fra^l \subseteq J$ and $J$ can be generated by $n$ elements, then 
    \[
        c_B^J(\fra) \leq nl.
    \]
\end{proposition}
\begin{proof}
    Since $\fra \subseteq \sqrt{J}$, we have $\fra^l \subseteq J$ for some $l$.  Write $J = (f_1, \dots, f_n)$.  Then for any finite extension $R \subseteq S \subseteq R^+$, we have that 
    \[
        (\fra S)_{> nl} \subseteq \overline{\fra^{nl} S} \subseteq \overline{J^n S} \subseteq JB.
    \]    
    Since this holds for all $R \subseteq S \subseteq R^+$, we see that $c^J_{B}(\fra) \leq nl$.  
\end{proof}

\begin{lemma}
    \label{lem.ParameterIdealBCMComputation}
    With notation as in \autoref{def.BCMThresholdGeneral}, assume that $B$ satisfies the Brian{\c{c}}on-Skoda property. Fix $x_1,\ldots, x_n$ a partial system of parameters of $R$ and set $J=(x_1,\ldots, x_n)$. Then, $c^J_B(J)=n$.
\end{lemma}
\begin{proof}
    By \autoref{prop.FinitenessAssumingBS}, $c^J_B(J)\leq n$.

    For the other inequality, given an integer $c>0$, we have that $(JB:_B x_1^{c-1 \over c}\cdots x_n^{c-1 \over c})=(x_1^{1 \over c},\ldots, x_n^{1\over c})B\neq B$ given that $x_1^{1 \over c},\ldots, x_n^{1\over c}$ is a {(permutable, as $B$ is balanced)} regular sequence on $B$. Thus, $x_1^{c-1 \over c}\cdots x_n^{c-1 \over c}\notin JB$. Since $x_1^{c-1 \over c}\cdots x_n^{c-1 \over c}\in(J R^+)_{n({c-1 \over c})}$, it follows that $(J R^+)_{n-\epsilon}\not\subseteq JB$ for all $\epsilon>0$. Therefore, $c^J_B(J)\geq n$.
\end{proof}

    We now obtain the following characteristic free analog of \cite[Theorem 3.3]{HunekeMustataTakagiWatanabeFThresholdsTightClosureIntClosureMultBounds}.

\begin{theorem}
    \label{thm.BCMVersionOfHMTWForParameter}
    Let $R$ be a complete Noetherian local domain and suppose $B$ is a \BCM $R^+$-algebra
    satisfying the Brian{\c{c}}on-Skoda property.  Fix $x_1,\ldots, x_n$ a
    partial system of parameters of $R$ and set $J=(x_1,\ldots, x_n)$.
    Given an ideal $I\supseteq J$, $c^J_B(I)=n$ if and only if $\overline{I} = \overline{J}.$
  \end{theorem}
  \begin{proof}
    If $\overline{I} = \overline{J}$, then 
    $c^J_B(I)=c^J_B(\overline{I})=c^J_B(\overline{J})=c^J_B(J)=n$ thanks to \autoref{lem.ParameterIdealBCMComputation}.
  
    Now suppose that $c^J_B(I)=n$. Then, we have from \autoref{eq.BCMThresholdIsAMin} that
    \[
        (I R^+)_{> n}
        \subseteq JB.
    \]
    Thus, using \autoref{lem.FractionalPowerOfIdealPowersAndProducts}, $(\overline{IR^+}) \cdot (I R^+)_{> n-1} = (I R^+)_{1}(I R^+)_{> n-1}\subseteq JB$ 
    and since $J\subseteq I$, we have that 
    $\overline{IR^+}  (J R^+)_{> n-1}\subseteq JB$; that is 
    \[
        \overline{I R^+} \subseteq (JB:_{R^+}(J R^+)_{> n-1}).
    \]
  
    Let $y\in \overline{I R^+}$. 
    Note that $x_1^{a_1\over t}\cdots x_n^{a_n \over t}\in (J R^+)_{> n-1}$ for all $a_1,\ldots, a_n, t\in\mathbb{Z}_{>0}$ such that 
    \[ 
        a_1+\cdots +a_n= t(n-1)+1.
    \] 
    We thus obtain that
    \[
        y\in (JB:_B x_1^{a_1\over t}\cdots x_n^{a_n \over t})=
        (x_1^{t-a_1\over t},\ldots, x_n^{t-a_n \over t})B
    \]
    for all $a_1,\ldots, a_n, t\in\mathbb{Z}_{>0}$ such that 
    $a_1+\cdots +a_n= t(n-1)+1$ and $a_i<t$ for all $i$.
    As a consequence, for all $t\in \mathbb{Z}_{>0}$, 
    \[
        y\in\bigcap_{c_1,\ldots,c_n}(x_1^{c_1\over t},\cdots, x_n^{c_n \over t})B
    \]
    where $(c_1,\ldots,c_n)$ runs through all $n$-tuples of strictly positive integers with $c_1 + \ldots + c_n =t-1$.
    By \cite[Lemma 3.5]{RodriguezSchwedeBrianconSkodaViaWeakFunctoriality} (which is really an argument of Hochster, \cf \cite[Section 3]{LipmanTeissierPseudoRational}), for $t > n$, the right side is equal to $(x_1^{1/t},\ldots, x_n^{1/t})^{t-n}B$ and so 
    $y\in (x_1^{1/t},\ldots, x_n^{1/t})^{t-n}B$.
    Additionally, since $(x_1^{1/t},\ldots, x_n^{1/t})^{t-n}\subseteq (JR^+)_{ 1 - {n \over t}}$, it follows that $y\in(JR^+)_{ 1 - {n \over t}}B\cap R^+$ for all $t\in \mathbb{Z}_{>n}$.
    
    We can now take a finite extension $R\subseteq S$, such that
    $y\in(JS)_{ 1 - {n \over t}}B\cap S.$
    By 2.3 in \cite{HochsterSolidClosure},
    $B$ is solid over $S$ and so
    $y\in\overline{(JS)_{ 1 - {n \over t}}}=  (JS)_{ 1 - {n \over t}}\subseteq (JR^+)_{ 1 - {n \over t}}.$
    Hence,
    \[y\in \bigcap_{t\in\mathbb{Z}_{>n}} (JR^+)_{ 1 - {n \over t}}\subseteq (JR^+)_{ 1}.\]
    Therefore,
    \[(I R^+)_{ 1}\subseteq (J R^+)_{ 1}\]
    After intersecting with $R$ and applying \autoref{lem.IntegralClosureExtensionFractionalPowersExtensionContraction}, we obtain $\overline{I}\subseteq \overline{J}$, which is what we wanted to show.
  \end{proof}

\section{Relation with \BCM-test ideals}

Our goal in this section is to relate our \BCM-thresholds to \BCM test ideals (as discussed in \autoref{subsec.BCMTestIdeals}).  
We first provide a variant of \BCM-test ideals which we will see essentially agrees with the previously developed test ideal theory from positive or mixed characteristic.

\begin{definition}[{\cf \cite{SeungsuLeeThesis}}]
    \label{def.tauBForValuative}
    Suppose $(R, \fram)$ is a complete local $\bQ$-Gorenstein normal domain.  Suppose $\frb \subseteq R^+$ is an ideal and $t \in \bQ_{\geq 0}$.  Suppose $B$ is a \BCM $R^+$-algebra.  We define 
    \[
        \tau_{B,\elt}(R, \frb) := \sum_{\phi} \sum_g \phi(gB)
    \]
    where the sum runs over $\phi \in \Hom_R(B, R)$ and $g \in \frb$.  

    In mixed characteristic, for $1 \gg \epsilon > 0$, we will also consider 
    \[ 
        \tau_{B,\elt}(R, p^{\epsilon} \frb) := \sum_{\phi} \sum_g \phi({p^{\epsilon}}gB)
    \]
    where $g$ and $\phi$ are as above.
\end{definition}

We will choose $\frb = (\fra R^+)_{>t}$ for some ideal $\fra \subseteq R$ and $t > 0$, in which case we will write $\tau_{B,\elt}(R, (\fra R^+)_{>t})$.  The subscript $\elt$ reminds us that we are doing our test ideal computation element-wise (compare with the definition found in \cite[Section 6.1]{HaconLamarcheSchwede}).

We show that these test ideals agree with common existing notions of test ideals in positive and mixed characteristic.

\begin{theorem}
    \label{thm.BCMTestIdealViaValuative}
    Suppose $(R, \fram)$ is a complete regular local ring, $\fra \subseteq R$ is an ideal, $t \geq 0$, and that $B$ is a \BCM $R^+$-algebra.  Then we always have the following:
    \begin{equation}
        \label{thm.BCMTestIdealViaValuative.eq.EasyContainment}
        \tau_B(R, \fra^{t+\epsilon}) \subseteq \tau_{B,\elt}(R, (\fra R^+)_{>t})
    \end{equation}
    for all $\epsilon > 0$.
    Additionally:
    \begin{enumerate}
        \item If $R$ is $F$-finite and  has positive characteristic $p$, then we have that 
            \[
                \tau(R, \fra^t) = \tau_B(R, \fra^{t}) = \tau_{B,\elt}(R, (\fra R^+)_{>t}).
            \]        
            \label{thm.BCMTestIdealViaValuative.PosChar}
        \item Suppose $R$ is of mixed characteristic and $B$ is perfectoid.  In this case, we have that
            \[
                \tau_B(R, p^{\epsilon}\fra^{t+\epsilon}) = \tau_{B,\elt}(R, p^\epsilon(\fra R^+)_{>t}).
            \]
            for $1 \gg \epsilon > 0$.  The $p^{\epsilon}$ and $+\epsilon$ on the left, as well as the $p^{\epsilon}$ on the right, can be removed if $B$ is sufficiently large.  That is, we have that:
            \[
                \tau_B(R, \fra^{t}) = \tau_{B,\elt}(R, (\fra R^+)_{>t}) \;\;\;\; \text{ for $B$ sufficiently large.}
            \]
            \label{thm.BCMTestIdealViaValuative.MixedChar}
          %
    \end{enumerate}    
\end{theorem}

Before beginning the proof, let us note why we need the perfectoid assumption in the mixed characteristic case.  The main reason is that we only know that any two BCM algebras $B_1, B_2$ can map to a larger $B$ in the case that the $B_i$ are perfectoid, and this makes the behavior in the perfectoid case much better.  Additionally, perfectoid algebras also behave in certain convenient ways under blowups.  It remains an open question whether any BCM $R$-algebra has a map to a perfectoid one, see \cite[Remark 4.10]{MaSchwedeSingularitiesMixedCharBCM}.

\begin{proof}
    Note we may take $\omega_R = R$ as $R$ is regular.
    The initial containment \autoref{thm.BCMTestIdealViaValuative.eq.EasyContainment} follows from the definitions as the sum defining the left side (\autoref{def.TestIdealsViaMatlisDuality} and \autoref{lem.UnDualDefinitionOfTestIdeal}) is a subset of the sum defining the right (\autoref{def.tauBForValuative}).  

    We now move to the positive characteristic case \autoref{thm.BCMTestIdealViaValuative.PosChar}.  We observe that $\tau(R, \fra^s) = \tau_{B}(R, \fra^{s})$ for any $s> 0$ by summing up $\tau(R, f^{1/n})$ over $f \in \fra^{\lceil sn\rceil}$ and utilizing \cite[Definition-Proposition 2.7 (Test ideals)]{MaSchwedeSingularitiesMixedCharBCM} and either \cite{TakagiFormulasForMultiplierIdeals} or \cite[Section 8.1, Remark 8.16]{BMPSTWW-RH}.  This proves the first equality $\tau(R, \fra^t) = \tau_B(R, \fra^{t})$.  
    The ideal $\tau_{B,\elt}(R, (\fra R^+)_{>t})$ is generated by $x = \phi(g)$ for $\phi \in \Hom_R(B, R)$ and $g \in (\fra R^+)_{>t} \subseteq R^+ \subseteq B$ and so fix such a $x = \phi(g)$.   Now, as we can write 
    \[
        \Hom_R(B, R) \cong \Hom_R(B \otimes_S S, R) \cong \Hom_S(B, \Hom_R(S, R)) = \Hom_S(B, \omega_S),
    \]
     we see that we can factor $\phi$ as $\phi : B \xrightarrow{\phi_S} \omega_S \xrightarrow{T_S} R$ for any finite extension $R \subseteq S \subseteq B$ where $T_S : \omega_S = \Hom_R(S, R) \to R$ is the evaluation-at-1 map and $\phi_S$ is $S$-linear.

    We know that $g$ is contained in $(\fra S)_{t + \epsilon}$ for some sufficiently large $S \supseteq R$ and any $1 \gg \epsilon > 0$.  
    It follows that we can write 
    \[
        x = \phi(g) = T_S(\phi_S(g)) = T_S(g \phi_S(1)).
    \]
    Furthermore, after assuming $t + \epsilon = n/m$ is rational if necessary, 
    we see that $g^m \in \overline{\fra^n S}$.
    This, plus the fact that $\phi_S(1) \in \tau_B(\omega_S) = \tau(\omega_S)$, implies the following, with justifications on right 
    \[ 
        \begin{array}{rl|l}
            x \in & T_S(g \tau(\omega_S)) \\
            = & T_S( \tau(\omega_S, g^1)) & (\text{\autoref{prop.BasicPropertiesOfTestIdeals} \autoref{prop.BasicPropertiesOfTestIdeals.EasySkoda}})\\
            = & T_S( \tau(\omega_S, (g^m)^{1/m})) & \text{($m/m = 1$, \autoref{prop.BasicPropertiesOfTestIdeals} \autoref{prop.BasicPropertiesOfTestIdeals.Unambiguity})}\\
            \subseteq & T_S(\tau(\omega_S, (\overline{\fra^n S})^{1/m})) & \text{($ (g^m) \subseteq \overline{\fra^n S}$, \autoref{prop.BasicPropertiesOfTestIdeals} \autoref{prop.BasicPropertiesOfTestIdeals.EasyContainments})} \\
            = & T_S(\tau(\omega_S, ({\fra^n S})^{1/m})) & (\text{\autoref{prop.BasicPropertiesOfTestIdeals} \autoref{prop.BasicPropertiesOfTestIdeals.IntegralClosureAgnostic}})\\   
            = & \tau(R, (\fra^n)^{1/m}) & (\text{\cite[Lemma 4.4]{SchwedeTuckerNonPrincipal}, \cf \cite{SchwedeTuckerTestIdealFiniteMaps}})\\
            \subseteq & \tau(R, \fra^{t+\epsilon}) & (\text{$t + \epsilon = n/m$, \autoref{prop.BasicPropertiesOfTestIdeals} \autoref{prop.BasicPropertiesOfTestIdeals.Unambiguity}})\\
            = & \tau(R, \fra^t) &  (\text{\autoref{prop.BasicPropertiesOfTestIdeals} \autoref{prop.BasicPropertiesOfTestIdeals.PlusEpsilonSmall}}).
        \end{array}
    \]
    This completes the proof in the characteristic $p > 0$ case.

    Finally, we handle \autoref{thm.BCMTestIdealViaValuative.MixedChar}, the mixed characteristic case.
    By and using the notation of \autoref{thm.MixedCharBlowupComparison}, for the first statement, it suffices to show that 
    \[
        \tau_{B,\elt}(R, p^{\epsilon} (\fra R^+)_{>t}) \subseteq \Image\Big( H^d_{\fram}(R) \to H^d_{\fram}(\myR \Gamma(X^+, \cO_{X^+}((t+\epsilon)G + \epsilon \Div_X p)))\Big)^{\vee}
    \]
    where again we interpret the dual as a submodule of $\omega_R = R$, and where
     $1 \gg \epsilon > 0$.  As $R$ is Noetherian and the left side  gets bigger as $\epsilon > 0$ gets smaller, we see that the left side is constant for sufficiently small $\epsilon > 0$.  Likewise, the right side is constant for sufficiently small $\epsilon > 0$.
        {The left side is the sum of the Matlis duals of images 
    \begin{equation}
        \label{eq.DifferentImages}
        \begin{array}{rcl}
            \Image\Big(H^d_{\fram}(R) \xrightarrow{\cdot gp^{\epsilon}} H^d_{\fram}(R^+) \Big) & \cong & \Image\Big(H^d_{\fram}(R) \xrightarrow{\cdot gp^{\epsilon}} H^d_{\fram}(\myR \Gamma(X^+, \cO_{X^+}))\Big)\\
             & \cong & \Image\Big(H^d_{\fram}(R) \to H^d_{\fram}(\myR \Gamma(X^+, \cO_{X^+}(\Div_{X^+} (g p^{\epsilon }))))\Big)
        \end{array}
    \end{equation}
    for $g \in  (\fra R^+)_{>t}$ where the first isomorphism is thanks to \cite[Lemma 4.8]{BMPSTWW-MMP}, see also  \cite{BhattAbsoluteIntegralClosure}.  Pick a finite $R \subseteq S \subseteq R^+$ with $S$ containing $g$.  Set $Y$ to be the normalization of $X$ in $K(S)$ and consider the natural map $h : Y \to X$.  As $g \in (\fra S)_{>t}$, for some $1 \gg \epsilon > 0$ we have that $\Div_Y g \geq (t + \epsilon) h^* G$ and so 
    \[
        \cO_{X^+}\big((t+\epsilon)\mu^* G + \epsilon \Div_{X^+}(p)\big) \subseteq \cO_{X^+}\big(\Div_{X^+}(gp^\epsilon)\big)
    \] 
    where $\mu : X^+ \to X$ is the canonical map (no rounding is needed as on $X^+$, arbitrary roots exist).
    Therefore each element from \autoref{eq.DifferentImages} is also in the image of the map from $H^d_{\fram}(\myR \Gamma(X^+, \cO_{X^+}((t+\epsilon)G + \epsilon \Div_X p)))$.  The result then follows by duality.  }

    For the second part of \autoref{thm.BCMTestIdealViaValuative.MixedChar}, by \cite[Theorem 8.11]{BMPSTWW-RH} if $B$ is sufficiently large, then the left side is unchanged if the $p^{\epsilon}$ and $+\epsilon$ are removed.  
    For the right side, notice that the image of $H^d_{\fram}(R) \xrightarrow{\cdot g} H^d_{\fram}(B)$ is Matlis dual to the image of $\Hom_R(B, R) \xrightarrow{\text{eval@}g} R$.  
    
    For each $g$, we claim we can choose $B_g$ such that the image of $\Hom_R(B_g, R) \xrightarrow{\text{eval@}g} R$ is minimal as $B_g$ ranges over sufficiently large perfectoid BCM $B$.  We explain this now.  Set $S_g := R[g] \subseteq R^+$.  We can apply \cite[Proposition 5.7]{MaSchwedeSingularitiesMixedCharBCM} so that $\tau_{B_g}(\omega_{S_g}) = \Image\big( \Hom_S(B_g, \omega_{S_g}) \xrightarrow{\text{eval@1}} \omega_{S_g}\big)$ is minimal, running over all perfectoid BCM $B$.  Note that $g\tau_{B_g}(\omega_{S_g}) = \tau_{B_g}(\omega_{S_g}, \Div g)$ from \cite[Lemma 6.6]{MaSchwedeSingularitiesMixedCharBCM}, and which agrees with the image $\Image\big( \Hom_S(B_g, \omega_{S_g}) \xrightarrow{\text{eval@g}} \omega_{S_g}\big)$ by unraveling the definition (also see for instance \cite[Definition 8.13]{BMPSTWW-RH}).  As above, as any map $B_g \to R$ factors through $\omega_{S_g}$, we see that the image of $\Hom_R(B_g, R) \xrightarrow{\text{eval@}g} R$ is minimal as claimed.

    Now, thanks to \cite[Theorem 4.9]{MaSchwedeSingularitiesMixedCharBCM} we can find a $B$ to which all $B_g$ map to over $R^+$.  
    By \cite[Lemma 5.6]{MaSchwedeSingularitiesMixedCharBCM} and Matlis duality, we can also pick $B$ so that in fact the image of 
    \[
        \Image\Big( p^{\epsilon} g \Hom_R(B, R) \xrightarrow{\text{eval@1}} R\Big) = \Image\Big( g \Hom_R(B, R) \xrightarrow{\text{eval@1}} R\Big).
    \]
    for $\epsilon > 0$ sufficiently small, potentially depending on $g$.  
    As we sum over $g \in (\fra R^+)_{>t}$, we may need to pick distinct $\epsilon > 0$.  However, as the target is Noetherian we only need finitely many $\epsilon > 0$ to stabilize the sum.  This completes the proof.
\end{proof}

We next recall a lemma, essentially due to Datta-Epstein-Tucker.  First though, for $B$ any $R$-module, we fix the notation 
\[ 
    I_B(J) := \{x \in B\;|\; \phi(xB) \subseteq J \text{ for all $\phi \in \Hom_R(B, R)$}\}.
\]
We observe that we always have $JB \subseteq I_B(J)$.  

\begin{lemma}[{\cite[Lemma 3.0.5]{RodriguezBCMThresholdsHypersurfaces}, \cf \cite[Corollary 4.3.15]{DattaEpsteinTuckerMittagLefflerAndFrobenius}}] \label{lemmaEqual}
	Let $J$ be an ideal of a complete regular local ring $(R, \fram)$, and $B$ a $\fram$-adically complete \BCM $R$-algebra. 
	We have that $I_{B}(J)=JB$.  {  Furthermore, even if $B$ is not necessarily $\fram$-adically complete, the same result holds if $J$ is $\fram$-primary.}
\end{lemma}
\begin{proof}
    First assume $B$ is $\fram$-adically complete.  Pick $x \in I_B(J)$.   As $R \to B$ is Ohm-Rush-Trace, we know that $x \in \big( \phi(x) \;|\; \phi\in\Hom_R(B,R)\big)B \subseteq JB$.  Thus $I_B(J) \subseteq JB$ and since the other containment is clear, we are done.

    {
        Now suppose instead that $J$ is $\fram$-primary but $B$ is not necessarily $\fram$-adically complete.  Take $x \in I_B(J)$.  Since we have already handled the case that $B$ is complete, we see that $x \in J\widehat{B}$, or in other words that $\overline{x} = 0 \in \widehat{B}/J\widehat{B} \cong B/JB$ and hence that $x \in JB$ as well.        
    }
\end{proof}

In an earlier version of this paper, we omitted the hypothesis that $B$ was $\fram$-adically complete.  We thank Neil Epstein for pointing out this gap.  

\begin{proposition}[\textit{c.f.} \cite{MustataTakagiWatanabeFThresholdsAndBernsteinSato, BlickleMustataSmithDiscretenessAndRationalityOfFThresholds}]\label{ThresholdsAndTest}
	Let $\fra, J$ be proper nonzero ideals of a complete regular local ring $R$ and $B$ a \BCM $R^+$-algebra. 
	\begin{enumerate}
		\item We have that 
        \[ 
            \tau_{B, \elt}(R,(\fra R^+)_{> c^J_B(\fra)})\subseteq J.
        \] 
        Hence, $\tau(R, \fra^{c^J_B(\fra)}) \subseteq J$ in positive characteristic and $\tau_{B}(R,\fra^{c^J_B(\fra)})\subseteq J$ for $B$ perfectoid and sufficiently large if $R$ is of mixed characteristic.
        \label{ThresholdsAndTesta}
		\item Additionally if $B$ is $\fram$-adically complete, then for {any real $\alpha > 0$}, $$c_B^{\tau_{B,\elt}(R, (\fra R^+)_{>\alpha})}\left(\fra\right)\leq \alpha.$$\label{ThresholdsAndTestb}
            Hence $c_B^{\tau(R, \fra^{\alpha})}\left(\fra\right)\leq \alpha$ in positive characteristic and $c_B^{\tau_{B}(R, \fra^{\alpha})}\left(\fra\right)\leq \alpha$ for $B$ perfectoid and sufficiently large if $R$ is of mixed characteristic.
	\end{enumerate}
\end{proposition}
\begin{proof}\,
	\begin{enumerate}
        \item We have that $(\fra R^+)_{>c_B^J(\fra)} \subseteq JB$ by \autoref{eq.BCMThresholdIsAMin}.  Hence we immediately see from the definition that $\tau_{B,\elt}(R, (\fra R^+)_{>c_B^J(\fra)}) \subseteq J$ as desired.  The further statements then follow from \autoref{thm.BCMTestIdealViaValuative}.
		
		\item 
        By our definition of $\tau_{B, \elt}$, we have that $(\fra R^+)_{>\alpha} B \subseteq I_{B}\big(\tau_{B,\elt}((\fra R^+)_{>\alpha})\big)$.
		As $B$ is $\fram$-adically complete, by \autoref{lemmaEqual},  $(\fra R^+)_{>\alpha}\subseteq \tau_{B,\elt}(R, (\fra R^+)_{>\alpha})B$.
		It follows that $c^{\tau_{B,\elt}((\fra R^+)_{>\alpha})}\left(\fra\right)\leq \alpha.$  The further statements again follow from \autoref{thm.BCMTestIdealViaValuative}.
	\end{enumerate}
\end{proof}

It follows from \autoref{def.FractionalIntegralClosureForRealInR+} that $\tau_{B,\elt}(R, (\fra R^+)_{>t}) = \bigcup_{\epsilon > 0} \tau_{B,\elt}(R, (\fra R^+)_{>t+\epsilon})$ where the union on the right, which is ascending as $\epsilon$ goes to zero, is over $\epsilon > 0$ such that $t +\epsilon$ is rational.  Hence by Noetherianity we have that $\tau_{B,\elt}(R, (\fra R^+)_{>t}) = \tau_{B, \elt}(R, (\fra R^+)_{>t+\epsilon})$ for $1 \gg \epsilon > 0$.  Thus we could ask about the behavior for $t - \epsilon$.  This leads us to the following definition inspired originally by \cite{EinLazSmithVarJumpingCoeffs}.

\begin{definition}
    Suppose $(R, \fram)$ is a complete Noetherian local ring.  For $B$ any $\fram$-adically complete \BCM $R^+$-algebra, we say that the \emph{$B_{\elt}$-jumping numbers of an ideal $\fra \subseteq R$} are the numbers $t > 0$ such that $\tau_{B, \elt}(R, (\fra R^+)_{>t}) \neq \tau_{B, \elt}(R, (\fra R^+)_{>t-\epsilon})$ for all $1 \gg \epsilon > 0$.
\end{definition}

Under moderate hypotheses, in characteristic $p > 0$ or mixed characteristic, these agree with other well known notions of jumping numbers in view of \autoref{thm.BCMTestIdealViaValuative}.

\begin{corollary}[\textit{c.f.} \cite{MustataTakagiWatanabeFThresholdsAndBernsteinSato, BlickleMustataSmithDiscretenessAndRationalityOfFThresholds}]\label{Main}
	Let $\fra$ be a proper nonzero ideal of a complete regular local ring $R$ and $B$ a $\fram$-adically complete \BCM $R^+$-algebra. The set of $B_{\elt}$-jumping numbers for $\fra$ is the same as the set of $B$-thresholds of $\fra$.  Hence, in mixed characteristic, for $B$ perfectoid and sufficiently large, the set of $B$-thresholds of $\fra$ is the same as the set of jumping numbers of $\tau_B(R, \fra^t)$.
\end{corollary}
\begin{proof}
	Suppose that $\alpha$ is a $B_{\elt}$-jumping number for $\fra$.
	By \autoref{ThresholdsAndTest}\autoref{ThresholdsAndTestb}, we know that $c_B^{\tau_{B, \elt}(R,(\fra R^+)_{>\alpha})}\left(\fra\right)\leq \alpha$ and, as a consequence, setting $c := c_B^{\tau_{B,\elt}(R,(\fra R^+)_{>\alpha})}(\fra)$, that
	$$\tau_{B,\elt}(R,(\fra R^+)_{>\alpha})\subseteq \tau_{B,\elt}\left(R,(\fra R^+)_{>c}\right).$$
	By \autoref{ThresholdsAndTest}\autoref{ThresholdsAndTesta}, 
    we have that 
    $$\tau_{B,\elt}\left(R,(\fra R^+)_{>c}\right)\subseteq \tau_{B,\elt}(R,(\fra R^+)_{>\alpha}).$$
	Hence, 
    \[ 
        \tau_{B,\elt}(R,(\fra R^+)_{>\alpha})= \tau_{B,\elt}\left(R,(\fra R^+)_{>c}\right).
    \]
	Given that $\alpha$ is a $B_{\elt}$-jumping number and $c = c_B^{\tau_{B,\elt}(R,(\fra R^+)_{>\alpha})}\left(\fra\right)\leq \alpha$, we see that then $c_B^{\tau_{B,\elt}(R,(\fra R^+)_{>\alpha})}\left(\fra\right)=\alpha$.
	
	Now redefine $c :=c_B^J(\fra)$ for some $J$ and assume that $\tau_{B,\elt}(R,(\fra R^+)_{>c})= \tau_{B,\elt}(R,(\fra R^+)_{>c'})$  for some $c'<c$.
	By \autoref{ThresholdsAndTest}\autoref{ThresholdsAndTesta}, $\tau_{B,\elt}(R,(\fra R^+)_{>c'})=\tau_{B,\elt}(R,(\fra R^+)_{>c_B^J(\fra)})\subseteq J$. 
	It follows from the definition of $\tau_{B,\elt}(R,(\fra R^+)_{>c'})$ that $(\fra R^+)_{>c'}\subseteq I_{B}(\tau_{B,\elt}(R,(\fra R^+)_{>c'}))$.
	By \autoref{lemmaEqual},  $(\fra R^+)_{>c'}\subseteq \tau_{B,\elt}(R,(\fra R^+)_{>c'})B$.
	Thus, $(\fra R^+)_{>c'}\subseteq\tau_{B,\elt}(R,(\fra R^+)_{>c'})B\subseteq JB$ and $c'\geq c_B^J(\fra)$.
	It follows that $c>c'\geq c_B^J(\fra)$, a contradiction.
	Therefore, $c$ is a $B_{\elt}$-jumping number for $\fra$.
\end{proof}

This also yields the following alternate characterization of $c^J_B(\fra)$ in the regular case.

\begin{corollary}
    \label{cor.CharacterizationOfThresholdViaTestIdeal}
    Assume the notation of \autoref{ThresholdsAndTest}, and assume additionally that $B$ is an $\fram$-adically complete $R^+$-algebra or $J$ is $\fram$-primary.  Finally, assume that $R$ is $F$-finite of characteristic $p > 0$ or of mixed characteristic with $B$ perfectoid and sufficiently large.  Then
    \[
        c^J_B(\fra) = \sup\{t > 0 \;|\; \tau_B(R, \fra^t) \nsubseteq J\} = \inf\{t > 0 \;|\; \tau_B(R, \fra^t) \subseteq J\}.
    \]
\end{corollary}
\begin{proof}
    First if $t > c^J_B(\fra)$ then $\tau_B(R, \fra^t) \subseteq J$ by \autoref{ThresholdsAndTest} \autoref{ThresholdsAndTesta}.  On the other hand if $t < c^J_B(\fra)$, then by definition $(\fra R^+)_{>t} \not\subseteq JB = I_B(J)$, and hence there exists $\phi : B \to R$ and $g \in (\fra R^+)_{>t}$ such that $\phi(gB) \not\subseteq J$.  But $\phi(gB) \subseteq \tau_B(R, \fra^t)$ which completes the proof.
\end{proof}

\section{Bounds on multiplicity}

We conclude the paper by proving an analog of \cite[Theorem 5.6]{HunekeMustataTakagiWatanabeFThresholdsTightClosureIntClosureMultBounds}.  We follow a detailed suggestion of Linquan Ma, who provided the key \autoref{prop.Ma}, to let us deform to the monomial case in characteristic $p > 0$.

We adopt the following notation.  
Suppose $(R, \fram)$ is a regular local ring of mixed characteristic $(0, p > 0)$ and $J = (x_1^{a_1}, \dots, x_d^{a_d})$.  Set $T =R[\fram t, t^{-1}]$ as the extended Rees algebra and $S = T/(t^{-1}) = \gr_{\fram}(R)$ the associated graded ring.  For any ideal $I \subseteq R$, we define  
    \[
        \initial(I) := \ker\big(S \to \gr_{\fram}(R/I)\big) \subseteq S.
    \]
    Note this is called $I^*$ for instance in \cite{ValabregaValla.FormRingsAndRegularSequences}.
    We note that $\initial(I)$ is an ideal of $S$, the associated graded ring of $R$. 
    Additionally, we define 
    \[
        I' = \dots \oplus I t^{-1} \oplus I \oplus (I \cap \fram)t \oplus (I \cap \fram^2)t^2 \oplus \dots \subseteq T.
    \]
    We observe that $I'$ is the largest  homogeneous ideal of $T$ contained in $(IT, t^{-1}-1)$ and that 
    \begin{equation}
        \label{eq.IPrimesViaColon}
        I'=IT:_T(t^{-1})^\infty.
    \end{equation} 
    To see this, let $\frb$ be the largest ideal  homogeneous ideal of $T$ contained in $(IT, t^{-1}-1)$. Note that for any homogeneous element $g\in T$, we have that $g\in (IT, t^{-1}-1)$ if and only if $gt^{-1}\in (IT, t^{-1}-1).$  It follows that $I'\subseteq IT:_T(t^{-1})^\infty \subseteq \frb.$ 
    Now, let $h\in \frb$ be homogeneous of degree $\ell\in\bZ$ and let $g=ht^{-\ell}.$ Since $g\in (IT, t^{-1}-1)$, we can write 
    \[
        g=\sum_{i=-n}^n f_i+ \left(\sum_{i=-n}^n h_i\right)(t^{-1}-1)
    \]
    for some $n\in\bZ_{>0}$, $f_i\in IT$ and $h_i\in T$ homogeneous of degree $i$. Thus,
    \[
        g=f_n+h_n+h_{-n}t^{-1}+\sum_{i=-n}^{n-1} \left(f_i+h_{i+1}t^{-1}-h_i\right)
    \]
    Since $g$ is homogeneous of degree zero, we have that $h_i\in IT$ for $i=n,-n$ and $h_{i+1}t^{-1}-h_i\in IT$ for all nonzero $i$ with $-n\leq i\leq n-1$. As a consequence, $h_i\in IT$ for all $i$, and so $g\in [IT]_0=I$.
    It follows that $h\in (\fram^\ell\cap I)t^\ell$ if $\ell> 0$ and $h\in It^{\ell}$ when $\ell\leq0$. Hence, $\frb\subseteq I'.$

We also observe that
    \begin{enumerate}
        \item  $I' T[t] = I T[t] = I R[t,t^{-1}]$ and that,
        \item  via the map $T \to S = T/(t^{-1})$, $I' S = \initial(I)$.  
    \end{enumerate}
Part (a) is clear.  For part (b), simply observe that $\initial(I)$ is kernel of 
\[  
\begin{array}{rl}
   & R/\fram \oplus \fram/\fram^2 \oplus \fram^2/\fram^3 \oplus \dots  \\
   \to &  R/I \oplus \fram/(I \cap \fram + \fram^2) \oplus \fram^2/(I \cap \fram^2 + \fram^3) \oplus \dots
\end{array}
\] 
which agrees with $I'S$.

    We make the following observation.  Since $R$ is regular, $S = k[X_1, \dots, X_d]$ where $k = R/\fram$ and $X_i$ corresponds to $x_i$.  Furthermore, we claim that 
    \begin{equation}
        \label{eq.DescriptionOfInitialJ}
        \initial(J) = (X_1^{a_1}, \dots, X_d^{a_d}).
    \end{equation}
    But since the initial forms\footnote{the initial form of $f \in \fram^n \setminus \fram^{n+1}$ is the residue of $f$ in $\fram^n/\fram^{n+1} = S_n$} of the $x_i^{a_i} \in S$ are clearly equal to $X_i^{a_i}$, and since $X_1^{a_1}, \dots, X_d^{a_d}$ form a regular sequence in $S$, we see that the claim holds by \cite[Proposition 2.1]{ValabregaValla.FormRingsAndRegularSequences}.

\begin{proposition}[Ma]
    \label{prop.Ma}
    Suppose now that $(R, \fram)$ is a regular local ring of mixed characteristic $(0, p> 0)$ which is essentially of finite type over a DVR and that $I \subseteq R$ is an ideal.    Set $T = R[\fram t, t^{-1}]$ the extended Rees algebra and $S = T/(t^{-1}) = \gr_{\fram}(R)$ the associated graded ring.  
    Then
    \[
        \tau(R, I^c)' \supseteq \tau(T, I'^c)
    \]
    where here $\tau$ denotes the test ideal of \cite{BMPSTWW-RH}.
\end{proposition}
\begin{proof}
    Note that $T[t] = R[t, t^{-1}]$ is a smooth extension of $R$.  Hence it follows as the test ideal of \cite{BMPSTWW-RH} commutes with localization and smooth maps (\cite[Theorem 8.41(k)]{BMPSTWW-RH}), that 
    \[
        \tau(R, I^c) T[t] = \tau(T[t], (IT[t])^c) = \tau(T, I'^c)T[t].
    \]
    As a consequence, we see that 
    \[
        \tau(T, I'^c) \subseteq \bigcup_n \tau(R, I^c) T : (t^{-1})^n
    \]
    But we saw that the right side is $\tau(R, I^c)'$ above in \autoref{eq.IPrimesViaColon}.
    %
\end{proof}

\begin{lemma}
    \label{lem.EasyLemOnGraded}
    With notation above, $\ell_R(R/\fra) = \ell_S(S/\initial(\fra))$.  
\end{lemma}
\begin{proof}
    We have that $S/\initial(\fra)$ is isomorphic to $\gr_{\fram}(R/\fra)$, and so $\ell_S(S/\initial(\fra))=\ell_S(\gr_{\fram}(R/\fra))=\ell_R(R/\fra)$.    
\end{proof}


We now come to our final result.  We follow the strategy laid out in \cite[Theorem 5.6]{HunekeMustataTakagiWatanabeFThresholdsTightClosureIntClosureMultBounds} again following the suggestion of Ma.

\begin{theorem}
    \label{thm.MultiplicityStatement}
Suppose $(R, \fram = (x_1, \dots, x_d))$ is a  $d$-dimensional regular local ring  which is an essentially finite type algebra over a DVR of mixed characteristic $(0, p> 0)$  and $J = (x_1^{a_1}, \dots, x_d^{a_d}) \subseteq R$ for some $a_i > 0$.  If $\fra$ is $\fram$-primary, then 
    \[
        e(\fra) \geq \Big( {d \over c^{J\widehat{R}}_B(\fra \widehat{R})} \Big)^d e(J)
    \]
    for $B$ a sufficiently large perfectoid \BCM $\widehat{R}^+$-algebra.  
\end{theorem}

\begin{proof}[Proof of \autoref{thm.MultiplicityStatement}]
    As $e(\fra) = \displaystyle\lim_{n \to \infty} {d! \over n^d }\ell_R(R/\fra^n)
    $ it suffices to show 
    \[
        \ell_R(R/\fra^n) \geq {n^d \over d!}\Big( {d \over c^J_B(\fra \widehat{R})} \Big)^d e(J) = {1 \over d!}\Big( {n d \over c^J_B(\fra \widehat{R})} \Big)^d e(J) = {1 \over d!}\Big( {d \over c^J_B(\fra^n \widehat{R})} \Big)^d e(J).
    \]
    Thus replacing $\fra^n$ by $\fra$, we simply must show that
    \[
        \ell_R(R/\fra) \geq {1 \over d!}\Big( {d \over c^J_B(\fra \widehat{R})} \Big)^d e(J).
    \]
    Indeed, in the proof of \cite[Theorem 5.6]{HunekeMustataTakagiWatanabeFThresholdsTightClosureIntClosureMultBounds} the authors made essentially the same reductions.
    

    Recall from \autoref{cor.CharacterizationOfThresholdViaTestIdeal}, and using that the formation of the test ideal on $R$ commutes with completion (\cite[Corollary 8.40]{BMPSTWW-RH}), we have that $c^J_B(\fra \widehat{R}) = \sup\{\lambda \;|\; \tau(R, \fra^\lambda) \nsubseteq J\} = \inf\{\lambda \;|\; \tau(R, \fra^\lambda) \subseteq J\}$ where $\tau$ is the test ideal of \cite{BMPSTWW-RH} since $B$ is large enough.
    Now, suppose $\lambda > c^J_B(\fra \widehat{R})$, then $\tau_B(R, \fra^\lambda) \subseteq J$.    
    Therefore, $\tau(R, \fra^\lambda)' \subseteq J'$ and hence from \autoref{prop.Ma} that 
    \[
        \initial J = J'S \supseteq \tau(R, \fra^{\lambda})'S \supseteq \tau(T, \fra'^{\lambda})S \supseteq \tau(S, (\fra'S)^{\lambda}) = \tau(S, \initial(\fra)^{\lambda})
    \]
    where the final containment is the restriction theorem for test ideals from mixed characteristic to positive characteristic,  see \cite[Theorem 5.9]{MaSchwedeSingularitiesMixedCharBCM} or \cite[Theorem 8.41(l)]{BMPSTWW-RH}.  Hence $c^{\initial J}(\initial \fra) \leq \lambda$.  
    Now, the proof of \cite[Theorem 5.6]{HunekeMustataTakagiWatanabeFThresholdsTightClosureIntClosureMultBounds} showed in fact the characteristic $p > 0$ statement that  
    \[
        \ell_S(S/\initial\fra) \geq {1 \over d!}\Big( {d \over c^{\initial{J}}(\initial{\fra})} \Big)^d e(\initial{J})
    \]
    by a direct computation of monomials.
    Therefore using this and \autoref{lem.EasyLemOnGraded}, we see that 
    \[
        \ell_R(R/\fra) = \ell_S(S/\initial{\fra}) \geq {1 \over d!}\Big( {d \over c^{\initial{J}}(\initial{\fra})} \Big)^d e(\initial{J}) \geq {1 \over d!}\Big( {d \over \lambda} \Big)^d e(\initial{J}).
    \]    
    Next note that $e(\initial(J)) = e(J)$ by \autoref{eq.DescriptionOfInitialJ} and hence
    $\ell_R(R/\fra) \geq  {1 \over d!}\Big( {d \over \lambda} \Big)^d e(J)$.
    Since this holds for all $\lambda > c^J_B(\fra \widehat{R})$ we are done by taking a limit.
\end{proof}

\bibliographystyle{skalpha}
\bibliography{main}

\end{document}